\documentclass{article}
\usepackage{graphicx,amsmath, amsfonts, amssymb,enumitem,amsthm} 
\usepackage[colorlinks,citecolor=blue,urlcolor=blue]{hyperref}
\usepackage{color}

\newcommand\Ex{{\mathbb E}}

\newcommand\Prob{{\mathbb P}}

\newcommand\Normal{{\mathcal N}}

\newcommand\Q{{\mathcal Q}}

\newcommand\R{{\mathbb R}}

\newcommand\dto{\overset{d}{\to }}
\newcommand\wto{\overset{w}{\to }}

\newcommand\tran{\top}

\newcommand\Pto{\overset{\Prob}{\to}}
\newcommand\PTto{\overset{\Prob_T}{\to}}

\newcommand\bra[1]{\langle #1 \rangle}

\DeclareMathOperator{\diag}{diag}

\newtheorem{theorem}{Theorem}[section]
\newtheorem{corollary}[theorem]{Corollary}
\newtheorem{lemma}[theorem]{Lemma}

\theoremstyle{definition}
\newtheorem{definition}[theorem]{Definition}

\theoremstyle{remark}
\newtheorem{remark}[theorem]{Remark}
\newtheorem{example}[theorem]{Example}

\newsavebox{\toy}
\savebox{\toy}{\framebox[0.65em]{\rule{0cm}{1ex}}}

\title{Beta Ensembles in the Freezing Regime and Finite Free Convolutions}

\author{
Fumihiko Nakano\footnote{Mathematical Institute, Tohoku University, Sendai,  Japan.
\newline Email: fumihiko.nakano.e4@tohoku.ac.jp}
\and
Khanh Duy Trinh\footnote{Global Center for Science and Engineering, Waseda University, Japan.
\newline
Email: trinh@aoni.waseda.jp }
\and 
Ziteng Wang \footnote{Graduate School of Fundamental Science and Engineering, Waseda University, Japan.
\newline
Email: ziteng@fuji.waseda.jp }
}

\begin{document}

\maketitle

\begin{abstract}
In the freezing regime where the system size $N$ is fixed and the inverse temperature $\beta$ tends to infinity, the eigenvalues of Gaussian beta ensembles converge to zeros of the $N$th Hermite polynomial. That law of large numbers has been proved by analyzing the joint density or reading off the random matrix model. This paper studies its dynamical version. We show that in the freezing regime, the eigenvalue processes called beta Dyson's Brownian motions converge to deterministic limiting processes which can be written as the finite free convolution of the initial data and zeros of Hermite polynomials. This result is a counterpart of those in the random matrix regime (when $N$ tends to infinity and the parameter $\beta$ is fixed) and the high temperature regime (when $N$ tends to infinity and $\beta N $ stays bounded). We also establish Gaussian fluctuations around the limit and deal with the Laguerre case.

		
\medskip
	
%
%
%
%
\end{abstract}

\section{Introduction}

Gaussian beta ensembles are a family of joint densities of the form 
\begin{equation}\label{GbE}
\frac{1}{Z_{N,\beta}}
    \prod_{1 \le i < j \le N} |\lambda_j - \lambda_i|^{\beta}
    \prod_{l=1}^N e^{-\beta\lambda_l^2/4},
    \qquad (\lambda_1 \le \cdots \le \lambda_N),
\end{equation}
where $N$ is the system size, $\beta > 0$ is the inverse temperature parameter and $Z_{N,\beta}$ is the normalizing constant. These densities generalize the joint density of eigenvalues of Gaussian Orthogonal/Unitary/Symplectic Ensembles (GOE, GUE and GSE). Based on the idea of tridiagonalizing Gaussian matrices of GOE or GUE, a tridiagonal random matrix model was  introduced in \cite{DE02}. Denote by 
\begin{align*}
	H_{N, \beta} 
	\sim
	\frac{1}{\sqrt{\beta}}\begin{pmatrix}
		\Normal(0,2)	&\chi_{(N-1)\beta}	\\
		\chi_{(N-1)\beta}	&\Normal(0,2)		&\chi_{(N-2)\beta}\\
		&	\ddots	&\ddots	&\ddots\\
			&&\chi_{\beta}	&\Normal(0,2)
	\end{pmatrix}
\end{align*}
the symmetric tridiagonal matrix consisting of independent entries, where the diagonal is an i.i.d. (independent identically distributed) sequence of random variables of the Gaussian distribution $\Normal(0,2)$, and the off diagonal $H_{N, \beta}(i, i+1)$ follows the chi distribution $\chi_{(N-i)\beta}$ with $(N-i)\beta$ degrees of freedom. Then the eigenvalues $\lambda_1 < \dots < \lambda_N$ of $H_{N, \beta}$ are distributed according to the Gaussian beta ensemble~\eqref{GbE}. Spectral properties of Gaussian beta ensembles in general, and  Wigner's semi-circle law and Gaussian fluctuations around the limit in particular, have been studied by analyzing the joint density or reading off the random matrix model \cite{DE06, Johansson98}.

Stochastic analysis has also been used to study Gaussian beta ensembles \cite[\S 4.3]{Anderson-book}. The objects are beta Dyson's Brownian motions, the strong solution of the following system of stochastic differential equations (SDEs)
\begin{equation}\label{eq:DBM}
d\lambda_i(t) = \sqrt{\frac{2}{\beta}}\, db_i(t)
+ \sum_{j:j \ne i} \frac{1}{\lambda_i(t) - \lambda_j(t)}\, dt, \quad \lambda_i(0) = a_i,
\quad (i = 1, \ldots, N),
\end{equation}
together with the constraint that $\lambda_1(t) \le \lambda_2(t) \le \cdots \le \lambda_N(t)$, almost surely. 
Here $\{ b_i(t) \}_{i=1}^N$ are independent standard Brownian motions.
The SDEs~\eqref{eq:DBM} have a unique strong solution \cite{Cepa-Lepingle-1997}. For $\beta \ge 1$, with probability one, the eigenvalue processes are non-colliding at any time $t > 0$ \cite{Graczyk-Malecki-2014}. Under the zero initial condition, that is, $\lambda_i(0) = 0, i = 1, \dots, N$, the joint distribution of $\{\lambda_i(t)/\sqrt t \}_{i=1}^N$ coincides with the Gaussian beta ensemble~\eqref{GbE}. A dynamical version of Wigner's semi-circle law and Gaussian fluctuations at the process level have been studied.

Now let us get into detail by introducing the limiting behavior of the empirical distributions. With two parameters, the system size $N$ and the inverse temperature $\beta$, three different regimes have been considered. For fixed $\beta > 0$, the empirical distribution 
\[
	L_{N, \beta} = \frac1N \sum_{i=1}^N \delta_{\lambda_i/\sqrt N}
\]
converges weakly to the standard semi-circle distribution, almost surely. Here $\delta_\lambda$ denotes the Dirac measure. This is Wigner's semi-circle law, which holds as long as $N \to \infty$ with $\beta N \to \infty$. We call it the random matrix regime. When $N \to \infty$ but $\beta N$ stays bounded, we are in a high temperature regime where the empirical distributions converge to a Gaussian-like distribution. The third regime, called the freezing regime, is when $N$ is fixed and $\beta$ tends to infinity. In this case, the eigenvalues converge to zeros of Hermite polynomials. The freezing regime is an intermediate step to investigate the random matrix regime \cite{DE06} (by letting $\beta \to \infty$ first, then letting $N \to \infty$). It was also used to identify the limiting measure in a high temperature regime by duality \cite{DS15}.

At the process level, the empirical measure process
\[
	\frac 1N \sum_{i = 1}^N \delta_{\lambda_i(t)/\sqrt N }
\] 
converges to a deterministic probability measure-valued process. In the random matrix regime, the limiting measure process is expressed as the free convolution of the initial measure and a semi-circle distribution \cite{Voit-Woerner-2022b}. Analogous results hold in the high temperature regime using the $c$-convolution defined in terms of the Markov--Krein transform \cite{NTT-2025}. The main goal of this paper is to express the limiting processes in the freezing regime as the finite free convolution of the initial data and zeros of Hermite polynomials. Besides, we also establish interesting limit theorems on the limiting behavior of the moment processes which are viewed as duality results to those in the high temperature regime.

Now we focus on the freezing regime. Letting $\beta \to \infty$, the random matrix $H_{N, \beta}$ converges in probability to a deterministic matrix 
\begin{equation}\label{Gauss-JN}
	J_{N} 
	=
	\begin{pmatrix}
		0	&\sqrt {N-1}	\\
		\sqrt {N-1}	&0	&\sqrt {N-2}\\
		&	\ddots	&\ddots	&\ddots\\
			&&\sqrt 1	&0
	\end{pmatrix}.
\end{equation}
Since the characteristic polynomial $\det(x - J_N)$ is nothing but the $N$th probabilist's Hermite polynomial $H_N(x)$, the eigenvalues of $J_N$ are the zeros $z_{1, N}^{(H)} < \dots < z_{N, N}^{(H)}$ of $H_N(x)$. Thus, in the freezing regime, by the continuity of roots of polynomials, 
\[
	\lambda_i \Pto z_{i, N}^{(H)}, \quad i = 1, \dots, N.
\]
Here `$\Pto$' denotes the convergence in probability.
Gaussian fluctuations have also been studied \cite{Andraus-HV-2021, Andraus-Voit-2019, DE05, Voit-2019}.

At the process level, formally, when the initial data $a_1 \le \cdots \le a_N$ are fixed, as $\beta \to \infty$, the eigenvalue processes $\lambda_1(t), \dots, \lambda_N(t)$ converge to deterministic processes $y_1(t) \le  \dots \le y_N(t)$ satisfying the following ordinary differential equations (ODEs)
\begin{equation}\label{ODEs-yi}
\left\{
\begin{aligned}
	y_i'(t) &= \sum_{j:j \ne i} \frac{1}{y_i(t) - y_j(t)}, \\
	y_i(0) &= a_i,
\end{aligned}
\right.
\quad i = 1, \dots, N.
\end{equation}
The existence and uniqueness of the solution have been shown in \cite{Voit-Woerner-2022}. Under the zero initial condition, the unique solution is given by
\[
	(y_1(t), \dots, y_N(t)) = \sqrt t (z_{1, N}^{(H)}, \dots, z_{N, N}^{(H)}).
\]
Our main result here is a type of Law of Large Numbers (LLN).

\begin{theorem}\label{thm:main-intro}
Let $\lambda_1^{(\beta)}(t) \le \cdots \le \lambda_N^{(\beta)}(t)$ be the unique strong solution to the system of SDEs~\eqref{eq:DBM}. (Here we have added the superscript ${}^{(\beta)}$ to indicate the dependence on $\beta$.) Then as $\beta \to \infty$, 
\[
	(\lambda_1^{(\beta)}(t), \dots, \lambda_N^{(\beta)}(t)) \Pto (a_1, \dots, a_N) \boxplus_N \sqrt t(z_{1, N}^{(H)}, \dots, z_{N, N}^{(H)}),
\]
uniformly for $t \in [0, T]$, where $\boxplus_N$ denotes the finite free convolution (see Definition~{\rm\ref{defn:FFC}}).
\end{theorem}

In the high temperature regime, the limiting behavior of the empirical measure process has been studied by a moment method \cite{NTT-2023, NTT-2025}. Gaussian fluctuations, or Central Limit Theorems (CLTs) involving orthogonal polynomials were established. In this work, we also use that moment method to deduce a new yet equivalent form of CLTs.

The paper is organized as follows. We briefly introduce the finite free convolution in the next section. The proof of Theorem~\ref{thm:main-intro} is then given in Sect.~\ref{sect:LLN}. Section~\ref{sect:EMP} establishes the LLN and the CLT for the empirical measure processes by using a moment method. Next, in Sect.~\ref{sect:Laguerre}, we deal with the Laguerre case (beta Laguerre ensembles and beta Laguerre processes). The paper ends with appendices on dual polynomials and finite free convolutions.

\section{Finite free convolution}
\label{sect:FFC}
Let us begin with defining a convolution of polynomials. For two polynomials of order $N$,
\[
	p(x) = \sum_{i=0}^N (-1)^i \alpha_i x^{N-i}, \quad  	q(x) = \sum_{i=0}^N (-1)^i \beta_i x^{N-i},
\] 
its $N$th symmetric additive convolution is defined to be 
\begin{equation}
	(p \boxplus_N q)(x) := \sum_{k=0}^N (-1)^k x^{N-k} \sum_{i+j=k} \frac{(N-i)!(N-j)!}{N!(N-k)!} \alpha_i \beta_{j}.
\end{equation}
That convolution which was introduced by Szeg\"o and Walsh in the 1920s has been rediscovered to be an expected characteristic polynomial of a sum of random matrices \cite{Marcus-Spielman-Srivastava-2022}.

It is known that if $p(x)$ and $q(x)$ are real rooted polynomials, so is $(p \boxplus_N q)(x)$. Thus, we define the finite convolution of two $N$-tuples of real numbers $(a_1, \dots, a_N)$ and $(b_1, \dots, b_N)$ to be the roots $(c_1, \dots, c_N)$ (in an ascending order) of $(p\boxplus_N q)(x)$, where
\[
	p(x) = \prod_{i=1}^N (x-a_i), \quad q(x) =\prod_{i=1}^N (x-b_i). 
\]
Note that 
\[
	p(x) = \prod_{i=1}^N (x-a_i) = \sum_{k=0}^N (-1)^k e_k(a_1, \dots, a_N) x^{N-k},
\]
where for $k = 1, \dots, N$,
\[
	e_k(x_1, \dots, x_N) = \sum_{1 \le j_1 < j_2 < \dots < j_k \le N} x_{j_1} x_{j_2} \cdots x_{j_k}
\]
are elementary symmetric polynomials in $N$ variables $x_1, \dots, x_N$, and $e_0(x_1, \dots, x_N) = 1$. Equivalently, we can define the finite free convolution in terms of elementary symmetric polynomials.

\begin{definition}\label{defn:FFC}
$(c_1, \dots, c_N)$ (in an ascending order) is said to be the finite free convolution of two $N$-tuples of real numbers $(a_1, \dots, a_N)$ and $(b_1, \dots, b_N)$,  denoted by
\[
	(c_1, \dots, c_N) = (a_1, \dots, a_N) \boxplus_N (b_1, \dots, b_N),
\]
if 
\begin{equation}\label{ekcab}
	e_k(c_1, \dots, c_N) = \sum_{i+j=k} \frac{(N-i)!(N-j)!}{N!(N-k)!} e_{i}(a_1, \dots, a_N) e_j(b_1, \dots, b_N), \quad k = 1, \dots, N.
\end{equation}
\end{definition}
\begin{remark}
For two discrete probability measures on $\R$ of the forms $\mu_a=\frac1N \sum_{i=1}^N \delta_{a_i}$ and $\mu_b=\frac1N \sum_{i=1}^N \delta_{b_i}$, their finite free convolution is the probability measure $\mu_c=\frac1N \sum_{i=1}^N \delta_{c_i}$, written as 
\[
	\mu_c = \mu_a \boxplus_N \mu_b,
\]
if $(c_1, \dots, c_N) = (a_1, \dots, a_N) \boxplus_N (b_1, \dots, b_N)$.
It was argued in \cite{Mergny-Potters-2022} that when $N \to \infty$, if $\mu_a$ (resp.\ $\mu_b$) converges weakly to $\mu_1$ (resp.\ $\mu_2$), then  $\mu_a \boxplus_N \mu_b$ converges weakly to the free convolution $\mu_1 \boxplus \mu_2$ of $\mu_1$ and $\mu_2$. That explains the terminology `finite free convolution'.
\end{remark}

We now introduce another explanation of the finite free convolution  \cite[\S 3.3]{Mergny-Potters-2022}.
To real numbers $a_1, \dots, a_N$, there are complex numbers $s_1, \dots, s_N$ such that 
\begin{equation}\label{atos}
	 \prod_{i=1}^N (z - a_i) = \frac1N \sum_{i=1}^N (z - s_i)^N.
\end{equation}
In other words, to a measure $\mu_a  = \frac1N \sum_{i=1}^N \delta_{a_i}$,
there is a measure $\nu_a = \frac1N \sum_{i=1}^N \delta_{s_i}$ supported on complex numbers $s_1, \dots, s_N$ such that 
\[
	\int (z - x)^N d\nu_a = \frac1N \sum_{i=1}^N (z - s_i)^N = \prod_{i=1}^N (z - a_i) = \exp\left(N \int \log(z - u) d\mu_a(u)\right).
\]
This relation is called the Markov--Krein relation with negative parameter $c=-N$~\cite{KreinNudelman1977}. Let
 $S_a$ be a complex-valued random variable distributed according to  
the discrete measure $\nu_a$. Since 
\[
	\Ex[S_a^k] = \frac 1N \sum_{i=1}^N s_i^k,
\]
it follows that the relation~\eqref{atos} is equivalent to the condition that 
\begin{equation}
	\binom N k \Ex[S_a^k] = e_k(a_1, \dots, a_N), \quad k = 1, \dots, N.
\end{equation}
Denote by $S_b$ and $S_c$ the corresponding random variables related to the probability measures $\mu_b = \frac1N \sum_{i=1}^N \delta_{b_i}$ and $\mu_c = \frac1N \sum_{i=1}^N \delta_{c_i}$, respectively. Then $\mu_c = \mu_a \boxplus_N \mu_b$, if and only if 
\[
	\Ex[S_c^k] = \sum_{i=0}^k \binom k i \Ex[S_a^i] \Ex[S_b^{k-i}], \quad k = 1, \dots, N.
\]
It tells us that the first $N$ moments of $S_c$ coincides with the corresponding moments of the sum of independent copies of $S_a$ and $S_b$. This explanation relates the finite free convolution with the concept of $c$-convolution, for $c > 0$, in the high temperature regime.  

We conclude this subsection with the following result on the convergence of elementary symmetric polynomials.

\begin{lemma}\label{lem:fromektox}
For each $\beta>0$, let $x_1^{(\beta)}\le \cdots \le x_N^{(\beta)}$ be real numbers. Assume that for every $k=1,\dots,N$,
\[
e_k\bigl(x_1^{(\beta)},\dots,x_N^{(\beta)}\bigr)
   \to c_k
   \quad\text{as} \quad \beta\to\infty.
\]
Then as $\beta \to \infty$ 
\[
	x_i^{(\beta)} \to x_i, \quad i = 1, \dots, N,
\]
where $x_1\le \cdots \le x_N$ are the zeros of the polynomial 
\[
	p(x) =  \sum_{k=0}^N (-1)^k c_k x^{N-k}.
\]
\end{lemma}
\begin{proof}
For each $\beta$, define the polynomial $p^{(\beta)}$ by 
\[
	p^{(\beta)}(x) = \prod_{i=1}^N (x - x_i^{(\beta)}) =  \sum_{k=0}^N (-1)^k e_k(x_1^{(\beta)},\dots,x_N^{(\beta)})x^{N-k}.
\]
By assumption, the coefficients of $p^{(\beta)}$ converge to the corresponding coefficients of $p$, which implies the convergence of $\{x_i^{(\beta)}\}$ by the continuity of the roots of polynomials \cite{Cucker-GonzalezCorbalan-1989}. The proof is complete.
\end{proof}

\section{Beta Dyson's Brownian motions in the freezing regime}
\label{sect:LLN}
Recall from the introduction that, formally, when $N $ is fixed and the initial data $a_1 \le a_2 \le \cdots \le a_N$ are fixed, as $\beta \to \infty$, beta Dyson's Brownian motions $\lambda_1^{(\beta)}(t), \dots, \lambda_N^{(\beta)}(t)$ converge to deterministic limits $y_1(t) \le \dots \le y_N(t)$ satisfying the ODEs~\eqref{ODEs-yi}.
Fix $N \ge 2$ from now on.
We are in a position to show that in the freezing regime, the eigenvalue processes converge to the finite free convolution of the initial data $(a_1, \dots, a_N)$ and $\sqrt t(z_{1, N}^{(H)}, \dots, z_{N, N}^{(H)})$. Here recall also that $z_{1, N}^{(H)}, \dots, z_{N, N}^{(H)}$ are zeros of the $N$th Hermite polynomial $H_N$, or the eigenvalues of the tridiagonal matrix $J_N$~\eqref{Gauss-JN}.

For $T > 0$, let $C([0, T])$ be the space of real continuous functions on $[0, T]$ endowed with the uniform norm 
\[
	\left\|f\right\|_\infty =  \sup_{0 \le t \le T} |f(t)|.
\]
We say that a sequence of $C([0, T])$-valued random elements $X^{(\beta)}(t)$ converges in probability to a deterministic limit $x(t) \in C([0, T])$, denoted by $X^{(\beta)}(t) \PTto x(t)$, if for any $\varepsilon > 0$,
\[
	\lim_{\beta \to \infty}\Prob(\| X^{(\beta) } - x\|_\infty \ge \varepsilon)  = 0.
\] 
For convenience, we restate Theorem~\ref{thm:main-intro} here.
\begin{theorem}
\label{thm:restate}
As $\beta \to \infty$, 
\[
	(\lambda_1^{(\beta)}(t), \dots, \lambda_N^{(\beta)}(t)) \Pto (a_1, \dots, a_N) \boxplus_N \sqrt t(z_{1, N}^{(H)}, \dots, z_{N, N}^{(H)}),
\]
uniformly for $t \in [0, T]$. More precisely, let 
\[
	(y_1(t), \dots, y_N(t)) =  (a_1, \dots, a_N) \boxplus_N \sqrt t(z_{1, N}^{(H)}, \dots, z_{N, N}^{(H)}).
\]
Then as $\beta \to \infty$, 
\[
\lambda_i^{(\beta)}(t) \PTto y_i(t), \quad i = 1, \dots, N.
\]
\end{theorem}

\begin{proof}
We divide the proof into several lemmata. The idea here is to investigate the limiting behavior of the elementary symmetric polynomials 
\[
e_k^{(\beta)}(t) := e_k(\lambda_1^{(\beta)}(t), \dots, \lambda_N^{(\beta)}(t)), \quad k = 0, \dots, N.
\] 
In Lemma~\ref{lem:ektogk}, we show that for each $k$, $e_k^{(\beta)}(t)$ converges uniformly for $t \in [0, T]$ to a deterministic process $g_k(t)$ defined recursively. Next, those uniform convergences of elementary symmetric polynomials imply that each $\lambda_i^{(\beta)}(t)$ converges uniformly to a limit $y_i(t)$  (Lemma~\ref{lem:ektolambdak}), where 
\[
	e_k(y_1(t), \dots, y_N(t)) = g_k(t), \quad k = 1, \dots, N.
\]
Finally, Lemma~\ref{lem:finite-convolution} states that the limits $(y_1(t), \dots, y_N(t))$ are the finite free convolution of the initial data $(a_1, \dots, a_N)$ and $\sqrt t(z_{1, N}^{(H)}, \dots, z_{N, N}^{(H)})$.
\end{proof}

Let us get into detailed arguments.
We begin with an application of It\^o's formula. Since $\frac{\partial^2 e_k}{\partial x_i^2} = 0$, it follows that for $k \ge 2$,
\begin{align}\label{dek}
    de_k^{(\beta)}(t)  &= de_k(\lambda_1^{(\beta)}(t), \dots, \lambda_N^{(\beta)}(t))\notag\\
    &= \sum_{i=1}^N \frac{\partial e_k}{\partial x_i}(\lambda_1^{(\beta)}(t), \dots, \lambda_N^{(\beta)}(t))\, d\lambda_i^{(\beta)}(t) \notag\\
    &= \sum_{i=1}^N \frac{\partial e_k}{\partial x_i}(\lambda_1^{(\beta)}(t), \dots, \lambda_N^{(\beta)}(t))\left( \sqrt{\frac{2}{\beta}}\, db_i(t)+ \sum_{\substack{j\neq i}}
            \frac{1}{\lambda_i^{(\beta)}(t)-\lambda_j^{(\beta)}(t)}\, dt \right) \notag\\
    &= \sqrt{\frac{2}{\beta}} \sum_{i=1}^N \partial_i e_k(\lambda_1^{(\beta)}(t)), \dots, \lambda_N^{(\beta)}(t))db_i(t) - \frac{(N-k+1)(N-k+2)}{2} e_{k-2}^{(\beta)}(t)dt.
\end{align}
Here for simplicity, we have used the notation $\partial_i := 
\frac {\partial }{\partial x_i}$ and the following identity.
\begin{lemma}
For $k \ge 2$ and for distinct $x_1, \dots, x_N$,
\[
	 \sum_{{j\neq i}}
\frac{\partial_i e_k(x_1, \dots, x_N)}{x_i - x_j} = - \frac{(N-k+1)(N-k+2)}{2} e_{k-2}(x_1, \dots, x_N).
\]
\end{lemma}
\begin{proof}
Regard the polynomial 
\[
	p = \prod_{l = 1}^N (x - x_l) = \sum_{k=0}^N (-1)^k e_k(x_1, \dots, x_N) x^{N-k}
\]
as a function of $x$ and $x_1, \dots, x_N$, we calculate its partial derivative
\[
	\partial_i p = -\prod_{l \neq i} (x - x_l) = \sum_{k=1}^N (-1)^k \partial_i e_k(x_1, \dots, x_N) x^{N-k}.
\]
It follows that for $i \neq j$,
\[
	\frac{\partial_i p - \partial_j p}{x_i - x_j} = - \prod_{l \neq i, j} (x - x_l) = \sum_{k=1}^N (-1)^k \frac{\partial_i e_k(x_1, \dots, x_N) - \partial_j e_k(x_1, \dots, x_N)}{x_i - x_j} x^{N-k}.
\]
Consequently by identifying the coefficient of $x^{N-k}$, we get that 
\[
	\frac{\partial_i e_k(x_1, \dots, x_N) - \partial_j e_k(x_1, \dots, x_N)}{x_i - x_j} = - e_{k-2}(\{x_1, \dots, x_N\} \setminus\{x_i, x_j\}).
\]
Finally, we use the usual `trick'
\begin{align*}
 \sum_{{j\neq i}} 
\frac{\partial_i e_k(x_1, \dots, x_N)}{x_i - x_j}  &= \frac12  \sum_{{j\neq i}}\frac{\partial_i e_k(x_1, \dots, x_N) - \partial_j e_k(x_1, \dots, x_N)}{x_i - x_j}\\
&=- \frac12  \sum_{{j\neq i}}^N e_{k-2}(\{x_1, \dots, x_N\} \setminus\{x_i, x_j\})\\
&= - \frac{(N-k+1)(N-k+2)}{2} e_{k-2}(x_1, \dots, x_N).
\end{align*}
Here the last line holds because for each $i_1 < \cdots < i_{k-2}$, the monomial $x_{i_1} \cdots x_{i_{k-2}}$ appears in $e_{k-2}(\{x_1, \dots, x_N\} \setminus\{x_i, x_j\})$ for $(N-k+1)(N-k+2)$ pairs of $(i, j)$. The proof is complete.
\end{proof}

We collect properties of the uniform convergence in probability in the following lemma whose proof is omitted.
\begin{lemma}
\label{2.2}
\quad
Let $\{X^{(\beta)}(t)\}_\beta$ and $\{Y^{(\beta)}(t)\}_\beta$ be two sequences of $C([0,T])$-valued random elements. 
Assume that there exist deterministic limit functions $x(t), y(t) \in C([0,T])$ such that as $\beta \to \infty$,
\[
X^{(\beta)}(t) \PTto x(t), 
\qquad 
Y^{(\beta)}(t) \PTto y(t).
\]
Then, for any $a(t), b(t) \in C([0,T])$, the following convergences hold.
\begin{enumerate}[label=\rm(\roman*)]
    \item 
    As $\beta \to \infty$,
    \[
    a(t) X^{(\beta)}(t) + b(t) Y^{(\beta)}(t)
    \PTto    a(t)x(t) + b(t)y(t).
    \]

    \item    
    As $\beta \to \infty$,
    \[
    X^{(\beta)}(t)\,Y^{(\beta)}(t)
    \PTto
    x(t)y(t).
    \]
    
    \item 
    As $\beta \to \infty$, 
\[
	\int_0^t X^{(\beta)}(s) ds	\PTto	\int_0^t x(s)ds.
\]
\end{enumerate}
\end{lemma}

For $k \ge 2$,
we write~\eqref{dek} in an integral form 
\begin{align*}
	e_k^{(\beta)}(t)  &= e_k(a_1, \dots, a_N) + \sqrt{\frac{2}{\beta}} \sum_{i=1}^N \int_0^t \partial_i e_k(\lambda_1^{(\beta)}(s)), \dots, \lambda_N^{(\beta)}(s))db_i(s) \\
	&\quad - \frac{(N-k+1)(N-k+2)}{2} \int_0^t e_{k-2}^{(\beta)}(s)ds.
\end{align*}
While when $k = 1$, 
\[
	e_1^{(\beta)}(t) = \sum_{i=1}^N  \lambda_i^{(\beta)}(t) = e_1(a_1, \dots, a_N) + \sqrt{\frac{2}{\beta}} \sum_{i = 1}^N b_i(t).
\]
Now for any $k \ge 1$, denote the martingale part by 
\[
K_k^{(\beta)}(t)
=\sqrt{\frac{2}{\beta}}
\sum_{i=1}^N \int_0^t \partial_i e_k(\lambda_1^{(\beta)}(s), \dots, \lambda_N^{(\beta)}(s))\, db_i(s).
\]
We show that the martingale part vanishes as $\beta \to \infty$.

\begin{lemma}
As $\beta \to \infty$, 
\[
	K_k^{(\beta)}(t) \PTto 0.
\]
\end{lemma}
\begin{proof}
For any given $\varepsilon > 0$, it follows from Doob's martingale inequality that
\[
	\Prob\left(\sup_{0 \le t \le T} |K_k^{(\beta)}(t)| \ge \varepsilon \right) \le \frac{1}{\varepsilon^2} \Ex[\bra{K_k^{(\beta)}}_T],
\]
where the quadratic variation $\bra{K_k^{(\beta)}}_T$ satisfies
\[
\langle K_k^{(\beta)}\rangle_T
=
\frac{2}{\beta} \int_0^T 
\sum_{i=1}^N \bigl|\partial_i e_k(\lambda_1^{(\beta)}(s), \dots, \lambda_N^{(\beta)}(s))\bigr|^2\, ds.
\]
Thus, it suffices to show that $ \Ex[\bra{K_k^{(\beta)}}_T] \to 0$ as $\beta \to \infty$.

Since each $\partial_i e_k(x_1, \dots, x_N)$ is a polynomial of degree $k-1$ in 
$(x_1,\dots,x_N)$, it follows that 
\[
\sum_{i=1}^N |\partial_i e_k(x_1, \dots, x_N)|^2
\le D_k\bigl(|x_1|^{2k-2} + \cdots + |x_N|^{2k-2}\bigr),
\]
holds for some constant $D_k$. In the next section, we show that 
\[
	\Ex \left[\frac1N\sum_{i=1}^N (\lambda_i^{(\beta)}(t) )^{2k-2}\right] \le C_{k-1, T}, \quad t \in [0, T],
\]
for some constant $C_{k-1, T}$ not depending on $\beta$. Therefore, 
\[
	 \Ex[\bra{K_k^{(\beta)}}_T]  \le \frac{const}\beta,
\]
which clearly tends to zero as $\beta \to \infty$. The proof is complete.
\end{proof}

Next, by induction, we deduce the following.
\begin{lemma}\label{lem:ektogk}
	As $\beta \to \infty$, 
\[
	e_k^{(\beta)}(t) \PTto g_k(t),
\]
where $g_k(t)$ is defined recursively by $g_0(t)=1, g_1(t)= e_1(a_1,\ldots,a_N),$
\begin{equation}\label{ODE-gk}
	g_k(t) = e_k(a_1, \dots, a_N) - \frac{(N-k+1)(N-k+2)}{2}  \int_0^t g_{k-2} (s) ds, \quad (k \ge 2).
\end{equation}
\end{lemma}
\begin{proof}
The proof follows immediately by induction since the martingale parts vanish uniformly on $[0,T]$.
\end{proof}
\begin{remark}
	The limit $g_k(t)$ can be characterized by the following ODEs
\begin{equation}\label{ODEsgk}
	g_k'(t) = -\frac{(N-k+1)(N-k+2)}{2}   g_{k-2}(t), \quad g_k(0) = e_k(a_1, \dots, a_N), \quad (k =2,\dots, N). 
\end{equation}
Here $g_0(t)=1$ and $g_1(t)= e_1(a_1,\ldots,a_N)$.
\end{remark}

Similar to Lemma~\ref{lem:fromektox}, the convergence of elementary symmetric polynomials implies the convergence of $\lambda_i^{(\beta)}(t)$ themselves.

\begin{lemma}\label{lem:ektolambdak}
Let $y_1(t) \le y_2(t) \le \cdots \le y_N(t)$ be defined from the relations
\[
	e_k(y_1(t), \dots, y_N(t)) = g_k(t), \quad k = 1, \dots, N.
\]
Then as $\beta \to \infty$, 
\[
	\lambda_i^{(\beta)}(t) \PTto y_i(t), \quad i = 1, \dots, N. 
\]
\end{lemma}

This lemma is a direct consequence of the following dynamical version of Lemma~\ref{lem:fromektox}. We omit the proof of Lemma~\ref{lem:ektolambdak}.

\begin{lemma}\label{lem:fromektox-process}
For each $\beta>0$, let $x_1^{(\beta)}(t) \le \cdots \le x_N^{(\beta)}(t)$ be continuous functions on $C([0,T])$.
Assume that for every $k=1,\dots,N$, there is a continuous function $h_k(t)$ such that
\[
\|e_k(x_1^{(\beta)}(t),\dots,x_N^{(\beta)}(t))
   - h_k(t)\|_\infty \to 0
   \quad\text{as} \quad \beta\to\infty.
\]
Then as $\beta \to \infty$, 
\[
	\|x_i^{(\beta)}(t) - x_i(t)\|_\infty \to 0, \quad i = 1, \dots, N,
\]
where $x_1(t)\le \cdots \le x_N(t)$ are the zeros of the polynomial 
\[
	 \sum_{k=0}^N (-1)^k h_k(t) x^{N-k}.
\]
\end{lemma}


\begin{proof}
It suffices to show that:
``for any 
$t \in [0,T]$, 
for any sequence 
$ \{t_{\beta}\}$
such that
$t_{\beta} 
\stackrel{\beta \to \infty}{\to} 
t$, 
we have 
$x_i^{(\beta)}(t_{\beta})
\stackrel{\beta \to \infty}{\to}
x_i(t)$". Take and fix such 
$t \in [0, T]$, 
and the sequence
$\{t_{\beta}\}$.
By the uniform convergence assumption, it holds that 
\[
e_k(x_1^{(\beta)}(t_{\beta}), \dots, x_N^{(\beta)}(t_{\beta}))
\to
h_k(t), \quad k = 1, \dots, N.
\]
Then Lemma 2.3 implies that
$x_i^{(\beta)}(t_{\beta})
\stackrel{\beta \to \infty}{\to}
x_i(t)$. The proof is complete.
\end{proof}

\begin{lemma}\label{lem:finite-convolution}
	The limiting processes $y_1(t), \dots, y_N(t)$ in Lemma~{\rm\ref{lem:ektolambdak}} are expressed as
\[
	(y_1(t), \dots, y_N(t)) = (a_1, \dots, a_N) \boxplus_N \sqrt t (z_{1,N}^{(H)}, \dots,  z_{N,N}^{(H)}).
\]
\end{lemma}
\begin{proof}
Let us first consider the zero initial condition, that is, $a_i = 0, i = 1, \dots, N$. Then the limits $g_k(t)$ can be explicitly calculated as 
\[
	g_0(t) = 1, \quad g_1(t) = 0,\quad  g_{2m+1}(t) = 0, 
\]
and 
\[
	g_{2m}(t) =  t^m \frac{(-1)^m}{2^m} \frac{N!}{m!(N-2m)!}.
\]
Then by definition, $y_1(t) \le \cdots \le y_N(t)$ are the zeros of the polynomial 
\[
	\sum_{2m \le N}  t^m \frac{(-1)^m}{2^m} \frac{N!}{m!(N-2m)!} x^{N-2m}.
\]
Note that the above sum at $t=1$ is exactly the $N$th probabilist's Hermite polynomial $H_N(x)$ (see Example~\ref{ex:Hermite}).  
We conclude that
\[
	(y_1(t), \dots, y_N(t)) = \sqrt t (z_{1,N}^{(H)}, z_{2,N}^{(H)},  \dots,  z_{N,N}^{(H)}).
\]

For general initial condition,
	let 
\[	
	(c_1(t), \dots, c_N(t)) = (a_1, \dots, a_N) \boxplus_N \sqrt t (z_{1,N}^{(H)}, z_{2,N}^{(H)},  \dots, z_{N,N}^{(H)}).
\]
By definition of the finite free convolution, 
\begin{align*}
	e_k(c_1(t), \dots, c_N(t)) &= \sum_{i=0}^k \frac{(N-i)!(N+i-k)!}{N!(N-k)!} e_{k-i}(a) e_i(z) t^{\frac i2}\\
	&= \sum_{0 \le m \le k/2}\frac{(N-2m)!(N+2m-k)!}{N!(N-k)!} e_{k-2m}(a) \frac{(-1)^m}{2^m m!} \frac{N!}{(N-2m)!}t^m\\
	&=\sum_{0 \le m \le k/2} \frac{(-1)^m}{2^m m!} \frac{(N+2m-k)!}{(N-k)!} e_{k-2m}(a)  t^m.
\end{align*}
Here for simplicity, $e_k(a)$ and $e_k(z)$ stand for $e_k(a_1, \dots, a_N)$ and $e_k(z_{1, N}^{(H)}, \dots, z_{N, N}^{(H)})$, respectively.
Take the derivative of $e_k(c_1(t), \dots, c_N(t))$, we get
\begin{align*}
	\frac{d e_k(c_1(t), \dots, c_N(t))}{dt} &=\sum_{1 \le m \le k/2} \frac{(-1)^m}{2^m m!} \frac{(N+2m-k)!}{(N-k)!} e_{k-2m}(a)  m t^{m-1} \\
	&= -\frac{(N-k+1)(N-k+2)}{2} e_{k-2}(c_1(t), \dots, c_N(t)), \quad k \ge 2.
\end{align*}

Note that $e_0(c_1(t), \dots, c_N(t)) = 1$, and 
$
	e_1(c_1(t), \dots, c_N(t)) = e_1(a).
$
We conclude that $e_k(c_1(t), \dots, c_N(t))$ satisfies the same ODE~\eqref{ODEsgk} as $g_k(t)$. Thus,
\[
	e_k(c_1(t), \dots, c_N(t)) = g_k(t) = e_k(y_1(t), \dots, y_N(t)), \quad k = 1,\dots, N,
\]
implying that 
\[
	y_i(t) = c_i(t), \quad i = 1, \dots, N.
\]
The proof is complete.
\end{proof}

\begin{remark}
It was shown in \cite{Voit-Woerner-2022} that the system of ODEs~\eqref{ODEs-yi} admits a unique solution which is exactly the limiting processes $y_1(t) \le y_2(t) \le \cdots \le y_N(t)$ here. 
Now, for zeros of Hermite polynomials, it is well-known that 
\[	
	\frac{1}{N}\sum_{i=1}^N \delta_{\sqrt t z_{i, N}^{(H)} / \sqrt N} \wto sc(t),
\]
where $sc(t)$ is the semi-circle distribution with variance $t$ whose density is given by  
\[
	\frac{1}{2\pi t}\sqrt{4t - x^2}, \quad |x| \le 2 \sqrt t, 
\]
and $\wto$ denotes the weak convergence of probability measures. Thus, Theorem~1.3 in \cite{fujie2025} implies that 
\[
	\frac{1}{N} \sum_{i=1}^N \delta_{y_i(t)/\sqrt N} \wto \mu_0 \boxplus sc(t),
\]
provided that $\sum_{i=1}^N \delta_{a_i/\sqrt N} \wto \mu_0$. Here `$\boxplus$' stands for the free convolution of probability measures on the real line. This provides another approach to show Theorem 1.1 in \cite{Voit-Woerner-2022b}.
\end{remark}

\section{Limit theorems for the empirical measure processes}
\label{sect:EMP}

In this section, we use a moment method to study the limiting behavior of the empirical measure process 
\[
	\mu_t^{(\beta)} = \frac1N\sum_{i=1}^N \delta_{\lambda_i^{(\beta)}(t)}
\]
of beta Dyson's Brownian motions \eqref{eq:DBM} starting at zero in the freezing regime.
We establish both the convergence to a deterministic limit and Gaussian fluctuations
around that limit. 

For $f = f(t,x) \in C^2((0,\infty) \times \mathbb{R})$, we write 
$\partial_t f = \frac{\partial f}{\partial t}$, 
$f' = \frac{\partial f}{\partial x}$, 
and $f'' = \frac{\partial^2 f}{\partial x^2}$. The starting point of our arguments is the following formula derived by using Itô's formula (cf.\ \cite{Cepa-Lepingle-1997, NTT-2023}),  
\begin{align}\label{dmutbeta}
d\langle \mu_t^{(\beta)}, f \rangle
&= \frac{1}{N} \sum_{i=1}^N d f(t,\lambda_i^{(\beta)}(t)) \notag \\[4pt]
&= \tfrac{\sqrt{2/\beta}}{N}\sum_{i=1}^N f'(t,\lambda_i^{(\beta)}(t))\,db_i(t)
+ \frac{N}{2}\iint \frac{f'(t,x) - f'(t,y)}{x - y}\, d\mu_t^{(\beta)}(x)\, d\mu_t^{(\beta)}(y)\, dt \notag \\[4pt]
&\quad + \left\langle \mu_t^{(\beta)}, \partial_t f + 
\left( \tfrac{1}{\beta} - \tfrac{1}{2} \right) f'' \right\rangle dt.
\end{align}
Here $\bra{\mu, f}$ denotes the integral $\int f(x) d\mu(x)$ of the integrable function $f$ with respect to the measure $\mu$, and for a function $f(t,x)$ of two variables $t$ and $x$, the integral $\langle \mu_t^{(\beta)}, f \rangle$ 
is taken over $x$. 
To be more precise, the above formula holds when $\lambda_1^{(\beta)}(t), \ldots, \lambda_N^{(\beta)}(t)$ are all distinct, which occurs almost surely when $\beta \ge 1$.

\subsection{Law of large numbers}
A moment method has been developed to show the LLN for the empirical measure processes $\mu_t^{(\beta)}$ in a high temperature regime. By using almost the same arguments as used in \cite{NTT-2025}, we can establish the LLN for moment processes of $\mu_t^{(\beta)}$, which implies the LLN for the empirical measure processes themselves. Under the zero initial condition, the main result is stated as follows.

\begin{theorem}
\label{thm:2.3}
The empirical measure process $\mu^{(\beta)}_t$ converges to a deterministic probability measure-valued process $\mu_t$ in probability under the topology
of uniform convergence in $[0, T]$, where
\[
	\mu_t = \frac1N \sum_{i=1}^N \delta_{\sqrt t z_{i, N}^{(H)}}.
\]
Moreover, for any polynomial $f(t,x)$ in $t$ and $x$, as $\beta \to \infty$,
\[
\langle \mu_t^{(\beta)}, f \rangle 
\PTto
\langle \mu_t, f \rangle.
\]
In addition, $\langle \mu_t, f \rangle$ is differentiable (as a function of $t$) 
and the following relation holds:
\begin{equation}\label{dtmutf}
\partial_t \langle \mu_t, f \rangle 
= \frac{N}{2} 
\iint \frac{f'(t,x) - f'(t,y)}{x - y} 
\, d\mu_t(x) d\mu_t(y)
+ \langle \mu_t,-\tfrac{1}{2} f'' + \partial_t f \rangle. 
\end{equation}
Note that we also use the partial derivative notation $\partial_t \langle \mu_t, f \rangle$ 
to denote the derivative with respect to $t$, though the function $\langle \mu_t, f \rangle$ 
depends only on $t$.
\end{theorem}
The LLN for empirical measure processes is a direct consequence of Theorem~\ref{thm:restate}. However, we are going to show it by a different approach, a moment method. The sketched proof is as follows.

\underline{Step 1} (Recursive relation for moment processes).
Denote by 
\[
	S_n^{(\beta)}(t) = \langle \mu_t^{(\beta)}, x^n \rangle 
= \frac{1}{N} \sum_{i=1}^N (\lambda_i^{(\beta)}(t))^n
\] 
the $n$th moment process of $\mu_t^{(\beta)}$. Formula~\eqref{dmutbeta} for $f(x) = x^n$ reads
\begin{align*}
dS_n^{(\beta)}(t)
&= \frac{n\sqrt{2/\beta}}{N} 
   \sum_{i=1}^N \lambda_i^{(\beta)}(t)^{n-1} db_i(t)
   + \frac{nN}{2} \sum_{j=0}^{n-2} 
     S_j^{(\beta)}(t) S_{n-2-j}^{(\beta)}(t)\, dt \notag \\[4pt]
&\quad + \left( \tfrac{1}{\beta} - \tfrac{1}{2} \right)
   n(n-1)\, S_{n-2}^{(\beta)}(t)\, dt, 
   \qquad (n \ge 2)
\end{align*}
or in the integral form
\begin{align}\label{Snbeta}
S_n^{(\beta)}(t) 
&=  \frac{n\sqrt{2/\beta}}{N} 
  \sum_{i=1}^N \int_0^t \lambda_i^{(\beta)}(u)^{n-1}\, db_i(u)+\frac{nN}{2} \int_0^t \sum_{j=0}^{n-2} 
   S_j^{(\beta)}(u)\, S_{n-2-j}^{(\beta)}(u)\, du \notag \\[6pt]
&\quad + \left( \frac{1}{\beta} - \frac{1}{2} \right)
   n(n-1) \int_0^t S_{n-2}^{(\beta)}(u)\, du,
   \qquad (n \ge 2).
\end{align}
Here the zero initial condition has been used.
For $n = 0,1$, it is clear that
\[
S_0^{(\beta)}(t) \equiv 1, \qquad 
S_1^{(\beta)}(t)
=\frac{\sqrt{2/\beta}}{N} 
  \sum_{i=1}^N b_i(t).
\]

\underline{Step 2} (Vanishing of the martingale parts). We also deal with the convergence of $C([0, T])$-valued random elements.
It is clear that $S_0^{(\beta)}(t) \equiv 1 \PTto 1$, and 
\[
	S_1^{(\beta)}(t) \PTto 0 \quad \text{as} \quad \beta \to \infty.
\]
The martingale parts
\[
M_n^{(\beta)}(t)
= \frac{n\sqrt{2/\beta}}{N} \sum_{i=1}^{N} 
  \int_{0}^{t} \lambda_i^{(\beta)}(u)^{n-1}\, db_i(u)
\]
are also $C([0, T])$-valued random elements. Their quadratic variations are given by 
\[
	\bra{M_n^{(\beta)}}_t = \frac{2n^2}{\beta N} \int_{0}^t \frac1N\sum_{i=1}^N \lambda_i^{(\beta)}(u)^{2(n-1)} du = \frac{2n^2}{\beta N} \int_{0}^t S_{2(n-1)}^{(\beta)}(u) du.
\]
We claim that there are constants $C_{n, T}$ depending on $n$ and $T$ such that for $t \in [0, T]$ and $\beta \ge 1$
\begin{equation}\label{S2nt}
	\Ex[S_{2n}^{(\beta)}(t)] \le C_{n, T}.
\end{equation}
We skip the proof of this claim because it is similar to the proof of equation~(25) in \cite{NTT-2025}. Then it follows from Doob's martingale inequality that 
\[
	\Prob\left(\sup_{0 \le t \le T} |M_n^{(\beta)}(t)| \ge \varepsilon\right) \le \frac{1}{\varepsilon^2} \Ex[\bra{M_n^{(\beta)}}_T] = \frac1{\beta} \frac{2n^2}{\varepsilon^2 N} \int_{0}^T \Ex[S_{2(n-1)}^{(\beta)}(u)] du \le   \frac1{\beta} \frac{2n^2}{\varepsilon^2 N} T C_{n-1, T}, 
\]
which clearly vanishes as $\beta \to \infty$. In conclusion, we have shown that 
\[
	M_n^{(\beta)}(t) \PTto 0\quad \text{as} \quad \beta \to \infty.
\]

\underline{Step 3} (The law of large numbers for moment processes). By induction, we arrive at the following LLN.
Define the functions $m_n(t)$ recursively by $m_0(t) \equiv 1, m_1(t) \equiv 0$, and for $n \ge 2$,
\begin{equation}\label{mnt}
m_n(t) = \int_0^t \left( 
\frac{nN}{2} \sum_{j=0}^{n-2} m_j(u) m_{n-2-j}(u)
- \frac{1}{2}n(n-1) m_{n-2}(u)
\right) du.
\end{equation}
Then 
\[
	S_n^{(\beta)}(t) \PTto m_n(t) \quad \text{as} \quad \beta \to \infty.
\]

\underline{Step 4} (Identifying the limiting measure processes).
By direct calculation, the limiting measure processes $m_n(t)$ have the following expression 
\[
	m_n(t) = u_n t^{n/2},
\]
where  $\{u_n\}$ satisfy
\begin{equation}\label{un}
\begin{cases}
u_{2n}
 &= -(2n-1)u_{\,2n-2} + N\displaystyle\sum_{j \le (n-1)} u_{2j} u_{\,2n-2-2j},\\
u_{2n+1}&=0.
\end{cases}
\end{equation}
The expression here is similar to that in the high temperature regime (Eq.\ (10) in \cite{NTT-2023}), which is viewed as a duality between two regimes. Let $v_n = u_{2n}, n \ge 0$. Then $\{v_n\}$ satisfies a self-convolutive recurrence as in \cite{Martin-Kearney-2010}. We conclude that
\[
	u_n = \frac1N \sum_{i=1}^N (z_{i, N}^{(H)})^n,
\]
and that $\{m_n(t)\}_{n \ge 0}$ are the moment processes of the probability measure-valued process
\[
	\mu_t = \frac{1}{N} \sum_{i=1}^N \delta_{\sqrt t z_{i, N}^{(H)}}.
\]

\underline{Step 5.} The convergence of moment processes implies the convergence of the empirical measure processes (see Theorem~A.1 in \cite{Trinh-Trinh-2021}). Finally, letting $\beta \to \infty$ in  equation~\eqref{dmutbeta}, we arrive at equation~\eqref{dtmutf}.

\begin{remark}
[Remark on duality and universality] (1) \emph{Duality.}
In the high temperature regime where $N \to \infty $ and $\beta = \frac{2c}{N}$,  the empirical measure $L_N = N^{-1} \sum_{i = 1}^N \delta_{\lambda_i}$ of the following Gaussian beta ensembles 
\begin{equation*}
const \times
    \prod_{1 \le i < j \le N} |\lambda_j - \lambda_i|^{\beta}
    \prod_{l=1}^N e^{-\lambda_l^2/2},
    \qquad (\lambda_1 \le \cdots \le \lambda_N),
\end{equation*}
converges weakly to a limiting measure $\rho_c$, almost surely \cite{Allez12, Trinh-2017}. Here $c > 0$ is given. The sequence $v_n$ of moments of $\rho_c$ satisfies a self-convolutive recurrence
\begin{eqnarray*}
v_{2n}
&=&
(2n-1) v_{2n-2}
+
c
\sum_{j \le n-1}
v_{2j}
\cdot
v_{2n-2 - 2j},\quad(v_{2n+1} = 0).
\end{eqnarray*}
Formally, results in the high temperature regime are obtained from those in the freezing regime by replacing $-N$ by $c$, plus some changes of signs. The reason for such duality is argued in \cite{DS15} as follows. Let 
\[
	m_n(N, \kappa) = \Ex[\bra{L_N, x^n}], \quad N = 1, 2, \dots; \kappa = \frac\beta2.
\]
Then $m_n(N, \kappa)$ is a polynomial in $N$ so that it can be defined for any $N \in \R$, and we have a duality relation that
\[
	m_n(N, \kappa) = (-1)^n \kappa^n m_n(-\kappa N, \kappa^{-1}).
\]
Therefore,, 
\[
	\lim_{N \to \infty, \kappa = c/N} m_n(N, \kappa) = \lim_{\kappa \to 0} (-1)^n \kappa^n m_n(-c, \kappa^{-1}).
\]
The left hand side is the limit in the high temperature regime while the right hand side is the limit in the freezing regime (with the system size $N = -c$).
This provides duality results between the two regimes.

Similarly, $\rho_c$ is the spectral measure of the following infinite Jacobi matrix 
\[
	A_{c} 
	=
	\begin{pmatrix}
		0	&\sqrt {c+1}	\\
		\sqrt {c+1}	&0	&\sqrt {c+2}\\
		&	\ddots	&\ddots	&\ddots\\
	\end{pmatrix}.
\]
That matrix is obtained from $J_N$ in \eqref{Gauss-JN} by replacing $-N$ with $c$ and changing the signs \cite{DS15}.

\medskip
(2) \emph{Universality.} We have explained in Section~\ref{sect:FFC} a relation between the finite free convolution and the $c$-convolution, for $c > 0$. Moreover, both the convolutions converge to the free convolution (as $N\to \infty$ and $c\to \infty$). The limiting measure process in any regime (random matrix regime, high temperature regime and freezing regime) can be written as the corresponding convolution of the initial measure and the limiting measure process under zero initial condition. We will see more similarity when looking at the formulation of CLTs in the next section.

\end{remark}

\subsection{Central limit theorems}
By arguments similar to those used in \cite{NTT-2023}, we can establish Gaussian fluctuations for the empirical measure processes. Let us highlight important points. 
Recall that the limiting measure $\mu_1$ ($\mu_t$ at $t = 1$) is expressed as 
\[
	\mu_1 = \mu^* =  \frac1N \sum_{i=1}^N \delta_{z_{i, N}^{(H)}}.
\]
Let $\{q_i\}_{i=0}^{N-1}$ be orthogonal polynomials with respect to $\mu^*$, the duals of the first $N$ Hermite polynomials, defined by the following three-term recurrence relation 
\begin{align*}
	&q_0 = 1, q_1 = x,\\
	&q_{n+1} = x q_n - (N-n) q_{n-1}, \quad n =1, \dots, N - 2.
\end{align*}
(See Example~\ref{ex:Hermite}.)
The orthogonal relation is expressed as
\begin{align*}
\bra{q_m, q_n}_{\mu^*} &:=
\int q_m(x)q_n(x)\,d\mu^*(x) \\
&= \frac1N \sum_{i=1}^N q_m(z_{i, N}^{(H)})q_n(z_{i, N}^{(H)})
  = \delta_{mn} \prod_{i=1}^{n} (N-i).
\end{align*}
By definition, $q_n$ is an odd polynomial (resp.\ even polynomial), if $n$ is odd (resp.\ even).
Let $Q_n$ be the primitive of $q_n$ with zero constant term. Define
\[
\widetilde{Q}_n(t,x) = t^{(n+1)/2} Q_n\!\left(\frac{x}{\sqrt{t}}\right).
\]
Then $\widetilde{Q}_n(t,x)$ is a polynomial in $t$ and $x$. The result on Gaussian fluctuations is stated as follows.

\begin{theorem}\label{theorem2.12}
As $\beta \to \infty$, the random processes
\[
\left\{
\sqrt{\beta N /2}\big( \langle \mu_t^{(\beta)}, \widetilde{Q}_n \rangle
- \langle \mu_t, \widetilde{Q}_n \rangle \big)
\right\}_{n=0}^{N-1}
\]
converge jointly in distribution  to independent centered Gaussian processes
$
\{\widetilde{\eta}_n(t)\}_{n=0}^{N-1},
$
with covariance given by

\begin{equation}\label{t25}
\mathbb{E}\!\left[\widetilde{\eta}_m(s)\widetilde{\eta}_n(t)\right]
= \delta_{mn}\,
\frac{\langle q_n, q_n \rangle_{\mu^*}}{n+1}\,
(s \wedge t)^{\,n+1}.
\end{equation}
\end{theorem}

\begin{corollary}
For Gaussian beta ensembles~\eqref{GbE}, which coincide with the joint distribution of $\{\lambda_i^{(\beta)}(1)\}_{i=1}^N$ under zero initial condition,  
as $\beta \to \infty$, 
\[
\left \{\sqrt{\beta N/2}\Big(\langle L_N, {Q}_n \rangle - \langle \mu^*, {Q}_n \rangle\Big) \right\}_{n=0}^{N-1}
\dto \Normal_N\left(0, \diag\left(\frac{\bra{q_n, q_n}_{\mu^*}}{n+1} \right)_{n=0}^{N-1} \right),
\]
where both $\diag(a_0, \dots, a_{N-1}),$ and $\diag(a_n)_{n=0}^{N-1}$ denotes the diagonal matrix with diagonal entries $a_0, \dots, a_{N-1}$, and $\Normal_N(0, \Sigma)$ denotes the $N$-dimensional Gaussian distribution with mean zero and covariance matrix $\Sigma$.
\end{corollary}

Let 
\[
	\hat q_n = q_n / \sqrt{\bra{q_n, q_n}_{\mu^*}}, \quad (n = 0, \dots, N-1),
\]
be orthonormal polynomials with respect to $\mu^*$.
Let $\hat Q_n$ be a primitive of $\hat q_n$. Then the above CLT can be re-written as 
\[
	\left\{\sqrt{\beta N/2}\Big(\langle L_N, \hat{Q}_n \rangle - \langle \mu^*, \hat{Q}_n \rangle\Big)\right\}_{n=0}^{N-1} \dto \Normal_N\left(0, \diag\left(\frac1{n+1}\right)_{n=0}^{N-1}\right).
\]

\begin{corollary}
\label{coro:G}
It holds that 
\[
	\{\sqrt{\beta /2} (\lambda_i^{(\beta)} - z_{i, N}^{(H)})\}_{i=1}^N \dto \Normal_N(0, \Sigma),
\]
where the limiting variance matrix $\Sigma$ satisfies 
\[
	\Q \Sigma \Q^{\tran} = \diag \left(\frac{1}{n+1} \right)_{n=0}^{N-1},
\]
with the orthogonal matrix 
\[
	\Q := \Big(\frac 1{\sqrt{N}} \hat q_n(z_{i, N}^{(H)}) \Big)_{n=0, \dots, N-1; i=1, \dots, N}.
\]
Consequently, for each $n = 0, \dots, N-1$, the limiting covariance matrix $\Sigma$ has a normalized eigenvector $\frac1{\sqrt N} (\hat q_n(z_{i, N}^{(H)}))_{i=1}^N$ with respect to the eigenvalue $1/(n+1)$.
\end{corollary}
\begin{proof}
Note that random vectors $(\lambda_i^{(\beta)})_{i=1}^N$ satisfy the following LLN and CLT:
\[
	(\lambda_i^{(\beta)})_{i=1}^N \Pto (z_{i, N}^{(H)})_{i=1}^N,
\]
\[
	\{\sqrt{\beta /2} (\lambda_i^{(\beta)} - z_{i, N})\}_{i=1}^N \dto \{\eta_i\}_{i=1}^N,
\]
where $\{\eta_i\}_{i=1}^N$ are jointly Gaussian random variables with mean zero and covariance matrix $\Sigma$. 

Let us express
\begin{align*}
\sqrt{\beta N/2}\left(\langle L_N, \hat{Q}_n \rangle - \langle \mu^*, \hat{Q}_n \rangle\right) &= \frac1{\sqrt N} \sum_{i=1}^N \sqrt {\beta /2} \left(\hat Q_n(\lambda_i^{(\beta)}) - \hat Q_n(z_{i,N}^{(H)})\right) \\
&=\frac1{\sqrt N} \sum_{i=1}^N  \hat q_n(\gamma_{i}^{(\beta)}) \sqrt {\beta /2}\left(\lambda_i^{(\beta)} - z_{i,N}^{(H)}\right),
\end{align*}
for some $\gamma_i^{(\beta)}$ between $\lambda_i^{(\beta)}$ and $z_{i, N}^{(H)}$, by applying the mean value theorem. Then the above LLN and CLT imply that 
\[
	\sqrt{\beta N/2}\left(\langle L_N, \hat{Q}_n \rangle - \langle \mu^*, \hat{Q}_n \rangle\right)  \dto \frac1{\sqrt N} \sum_{i=1}^N  \hat q_n(z_{i,N}^{(H)})\eta_i.
\]
The joint convergence also holds. Consequently, as column vectors, 
\[
	\left\{\sqrt{\beta N/2}\left(\langle L_N, \hat{Q}_n \rangle - \langle \mu^*, \hat{Q}_n \rangle\right)\right\}_{n=0}^{N-1} \dto  \Q \eta, \quad \eta = (\eta_1, \dots, \eta_N)^\tran.
\]
Therefore, the covariance matrix $\Q \Sigma \Q^\tran$ of $\Q \eta$ coincides with $\diag(\frac1{n+1})_{n=0}^{N-1}$, completing the proof.
\end{proof}

\section{Beta Laguerre ensembles and beta Laguerre processes}
\label{sect:Laguerre}

\subsection{Law of large numbers for beta Laguerre ensembles}
Consider beta Laguerre ensembles which are generalizations of Wishart matrices and Laguerre matrices with the following joint density 
\begin{equation}\label{bLE}
	\frac{1}{Z_{N, \alpha}^{(\beta)}}\prod_{i < j}|\lambda_j - \lambda_i|^\beta \prod_{i = 1}^N \left( \lambda_i^{\frac{\beta}{2}\alpha - 1} e^{-\frac{\beta}2  \lambda_i} \right), \quad (0 < \lambda_1 < \cdots < \lambda_N),
\end{equation}
where $\alpha > 0$ and $\beta > 0$ are two parameters, and ${Z_{N, \alpha}^{(\beta)}}$ is the normalizing constant. 
They are realized as the eigenvalues of the following tridiagonal matrix \cite{DE02}
\[
	L_{N,\beta} = (B_N^{(\beta)})^\tran B_N^{(\beta)},
\]
with bidiagonal matrix $B_N^{(\beta)}$ consisting of independent random variables distributed as follows
\[
	B_N^{(\beta)} = \frac{1}{\sqrt{\beta}} \begin{pmatrix}
		\chi_{\beta (\alpha + N - 1)}	\\
		\chi_{\beta (N-1)}	&\chi_{\beta(\alpha + N - 2)}	\\
		&\ddots	&\ddots\\
		&&\chi_\beta	&\chi_{\beta \alpha}
	\end{pmatrix}.
\]

When $N$ and $\alpha > 0$ are fixed, as $\beta \to \infty$,
\begin{align*}
	B_N^{(\beta)} 
	&\Pto
	\begin{pmatrix}
		\sqrt{\alpha + N - 1}	\\
		\sqrt{N-1}	&\sqrt{\alpha + N - 2}	\\
		&\ddots	&\ddots\\
		&&1	&\sqrt{\alpha}
	\end{pmatrix} =: B_N^{(\infty)}.
\end{align*}
Consequently, in the freezing regime where $N$ and $\alpha$ are fixed, and $\beta \to \infty$, the eigenvalues $(\lambda_1, \dots, \lambda_N)$ converge in probability to the eigenvalues of the deterministic tridiagonal matrix 
\begin{equation}\label{JNL}
	J^{(L)}_{N} := (B_N^{(\infty)})^\tran (B_N^{(\infty)}).
\end{equation}
Those deterministic eigenvalues turn out to be the zeros of the $N$th Laguerre polynomial $L_N^{(\alpha)}(x)$. Here Laguerre polynomials $\{L_n^{(\alpha)}(x)\}_{n \ge 0}$ are monic polynomials orthogonal with respect to the weight $x^{\alpha - 1} e^{-x}, x > 0$ (see Example~\ref{ex:Laguerre}). These arguments lead to the following LLN.
\begin{theorem}
	In the freezing regime, the eigenvalues of beta Laguerre ensembles~\eqref{bLE} converge in probability to the zeros of the $N$th Laguerre polynomial $L_N^{(\alpha)}(x)$. 
\end{theorem}

\subsection{Law of large numbers for beta Laguerre processes}
The so-called beta Laguerre processes $0 \le \lambda_1(t) \le \lambda_2(t) \le \cdots \le \lambda_N(t)$ satisfy the following system of 
SDEs
\begin{equation}\label{SDEs-L}
	d\lambda_i(t)
= \frac2{\sqrt\beta} \sqrt{{\lambda_i(t)}}\, db_i(t)
   + \alpha\, dt
   + \displaystyle\sum_{j \neq i}
     \frac{2\lambda_i(t)}{\lambda_i(t) - \lambda_j(t)}\, dt,
\end{equation}
with initial condition $0 \le a_1 \le a_2 \le \cdots \le a_N$. The existence and uniqueness of the strong solution to the above SDEs have been shown in \cite{Graczyk-Jacek-2014} when $\beta \ge 1$. In this case, the eigenvalue processes are non-colliding at any positive time. For more general $\beta>0$, beta Laguerre processes are defined to be the squared of type B radial Dunkl processes \cite{Demni-2007-arxiv}. The very first result in the freezing regime is stated as follows.

\begin{theorem}
\label{thm:LLN-L1}
	Let $\alpha > 0$ and $N\ge 2$ be fixed. Then as $\beta \to \infty$, 
\[
	\lambda_i(t) \PTto x_i(t), \quad i = 1, \dots, N,
\]
where the deterministic limiting processes $x_1(t),\dots, x_N(t)$ depend on the parameter $\alpha > 0$.
\end{theorem}
\begin{proof}
Let us outline main ideas in the proof. Detailed arguments are omitted because they are similar to those used in the Gaussian case. We begin with investigating the elementary symmetric polynomials of the eigenvalue processes,
\[
	e_k^{(\beta)}(t) = e_k(\lambda_1(t), \dots, \lambda_N(t)), \quad k = 0, \dots, N.
\]
First, by It\^o's formula, we obtain 
\[
d e_k^{(\beta)}(t) 
= \frac2{\sqrt \beta} \sum_{i=1}^N 
\frac{\partial e_k}{\partial \lambda_i}(\lambda_1(t), \dots, \lambda_N(t))
\sqrt{{\lambda_i(t)}}\,
db_i(t)
+ (N-k+1)(N-k+\alpha)\, e_{k-1}^{(\beta)}(t)\, dt.
\]
Next, the martingale parts vanish in the limit when $\beta \to \infty$. 
Then by induction, it follows that for $k = 1, \dots, N,$
\[
	e_k^{(\beta)}(t)  \PTto g_k(t),
\]
where $g_k(t)$ is defined recursively by 
\begin{equation}\label{gk-L}
g_k(t)= e_k(a_1, \dots, a_N) + (N-k+1)(N-k+\alpha) \int_0^t g_{k-1}(u)du, \quad (g_0(t)\equiv 1).
\end{equation}
Finally the limiting processes $x_1(t) \le \dots \le  x_N(t)$ are defined from the relations 
\[
	e_k(x_1(t), \dots, x_N(t)) = g_k(t), \quad k = 0, \dots, N.
\]
This concludes the sketched proof.
\end{proof}

\begin{remark}
The existence and uniqueness of the solution to the system of ODEs 
\[
	\frac{dz_i(t)}{dt} = \frac{\alpha}{z_i(t)} + \sum_{j \neq i} \left( \frac{1}{z_i(t) - z_j(t)} + \frac{1}{z_i(t) + z_j(t)}\right), \quad i = 1, \dots, N,
\]
have been proved in \cite{Voit-Woerner-2022}. Formally, the limiting processes $x_1(t),\dots, x_N(t)$ satisfy 
\[
	\frac{dx_i(t)}{dt} = \alpha + \sum_{j \neq i} \frac{2x_i(t)}{x_i(t) - x_j(t)}, \quad i = 1, \dots, N,
\]
which can be obtained from the ODEs of $z_i(t)$ by the change of variables $x_i(t) = \frac12 z_i(t)^2$.
\end{remark}

Next, we are going to study those limiting processes in more detail. The first property is related to the fact that starting at zero, at any time $t > 0$, the joint distribution of $(\lambda_1(t), \dots, \lambda_N(t))/t$ coincides with the beta Laguerre ensemble~\eqref{bLE}. 
\begin{lemma}\label{lem:Claim1}
Under the zero initial condition, the limiting processes are given by
\[
	(x_1(t), \dots, x_N(t)) = t (z_{1, N}^{(\alpha)}, \dots, z_{N, N}^{(\alpha)}).
\]
Here $0 < z_{1, N}^{(\alpha)} < \dots < z_{N, N}^{(\alpha)}$ are the zeros of the $N$th Laguerre polynomial $L_N^{(\alpha)}$.
\end{lemma}

\begin{proof}
Under the zero initial condition, the limiting processes $g_k(t)$ in the proof of Theorem~\ref{thm:LLN-L1} are explicitly calculated as 
\[
	g_k(t) =  \frac{t^k}{k!}\prod_{j=0}^{k-1} (N-j)(N-j +\alpha - 1), \quad k = 1, \dots, N.
\]
Then, we deduce from the definition that $x_i(t)$ have the form 
\[
	x_i(t) = t z_i, \quad i = 1, \dots, N,
\]
where $z_1, \dots, z_N$ are zeros of the following polynomial
\[
	\sum_{k=0}^N \frac{(-1)^k}{k!}\prod_{j=0}^{k-1} (N-j)(N-j +\alpha - 1) x^{N-k},
\]
which is nothing but the $N$th Laguerre polynomial $L_N^{(\alpha)}(x)$ (see formula~\eqref{closedform} in Appendix~\ref{sect:duals}). The proof is complete.
\end{proof}

\begin{lemma}\label{lem:Claim2}
Let $(x_1^{(i)}(t), \dots, x_N^{(i)}(t))$ be the limiting processes with parameter $\alpha_i, (i = 1, 2)$. Then 
\[
	(z_1(t), \dots, z_N(t)) := (x_1^{(1)}(t), \dots, x_N^{(1)}(t)) \boxplus_N (x_1^{(2)}(t), \dots, x_N^{(2)}(t))
\]
are the limiting processes with parameter $\alpha_1 + \alpha_2 + N - 1$.
\end{lemma}
\begin{proof}
	For simplicity, let 
\[
	c_k = e_k(x_1^{(1)}(t), \dots, x_N^{(1)}(t)), \quad d_k = e_k(x_1^{(2)}(t), \dots, x_N^{(2)}(t)), \quad f_k = e_k(z_1(t), \dots, z_N(t)). 
\]
We note from the definition of $g_k(t)$ in the proof of Theorem~\ref{thm:LLN-L1} that $c_k$ and $d_k$ are characterized by 
\[
	c_k' = (N-k+1)(N-k+\alpha_1)c_{k-1}, \quad 	d_k' = (N-k+1)(N-k+\alpha_1)d_{k-1}.
\]
Here `$'$' denotes the derivative with respect to $t$.
Then by definition of the finite free convolution,
\[
	f_k = \sum_{i+j=k} \frac{(N-i)!(N-j)!}{N!(N-k)!} c_i d_j.
\]
Take the derivative of both sides, we obtain that 
\[
	f_k' = \sum_{i+j=k} \frac{(N-i)!(N-j)!}{N!(N-k)!} \left\{c_i' d_j + c_i d_j' \right\}.
\]
Let us consider the first part 
\begin{align*}
	&\sum_{i+j=k} \frac{(N-i)!(N-j)!}{N!(N-k)!} c_i' d_j \\
	&=\sum_{i+j=k, i \ge 1} \frac{(N-i)!(N-j)!}{N!(N-k)!} (N-i+1)(N-i+\alpha_1)c_{i-1} d_j \\
	&=\sum_{\hat i + j = k-1} \frac{(N-\hat i)!(N-j)!}{N!(N-(k-1))!}(N-k+1) (N-\hat i+1+\alpha_1)c_{\hat i} d_j \quad (\hat i := i - 1) \\
	&=\sum_{i + j = k-1} \frac{(N- i)!(N-j)!}{N!(N-(k-1))!}(N-k+1) (N- i+1+\alpha_1)c_{ i} d_j \quad (i :=\hat i). 
\end{align*}
Similarly, the second part becomes 
\begin{align*}
	&\sum_{i+j=k} \frac{(N-i)!(N-j)!}{N!(N-k)!} c_i d_j' \\
	&=\sum_{i + j = k-1} \frac{(N- i)!(N-j)!}{N!(N-(k-1))!}(N-k+1) (N- j+1+\alpha_2)c_{ i} d_j. 
\end{align*}
By combining the two parts, we arrive at 
\begin{align*}
	f_k' &= \sum_{i+j=k} \frac{(N-i)!(N-j)!}{N!(N-k)!} \left\{c_i' d_j + c_i d_j' \right\}\\
	&=\sum_{i + j = k-1} \frac{(N- i)!(N-j)!}{N!(N-(k-1))!}(N-k+1) \left\{(N- i+1+\alpha_1)+(N- j+1+\alpha_2) \right\}c_{ i} d_j\\
	&=(N-k+1)(N-k+(\alpha_1+\alpha_2+N-1))\sum_{i + j = k-1} \frac{(N- i)!(N-j)!}{N!(N-(k-1))!}c_{ i} d_j\\
	&=(N-k+1)(N-k+(\alpha_1+\alpha_2+N-1))f_{k-1}.
\end{align*}
This implies that $z_1(t), \dots, z_N(t)$ are the limiting processes with parameter $\alpha_1 + \alpha_2 + N -1$. The proof is complete.
\end{proof}

\begin{lemma}\label{lem:Claim3}
{\rm(i)} Let $(y_1(t), \dots, y_{2N}(t))$ be the limiting processes in the Gaussian case with symmetric initial condition $a_1 \le a_2 \le \cdots \le a_{2N}$,
that is,
\[
	(y_1(t), \dots, y_{2N}(t)) = (a_1, \dots, a_{2N}) \boxplus_{2N} \sqrt{t} (z_{1,2N}^{(H)}, \dots, z_{2N, 2N}^{(H)}).
\]
Here $(a_1 \le a_2 \le \cdots \le a_N)$ are symmetric if 
\[
a_i = - a_{N-i+1}, \quad 1 \le i \le N.
\]
Then $(y_1(t), \dots, y_{2N}(t))$ are also symmetric for $t \ge 0$, and 
\[
	x_i(t) = \frac12 y_{N+i}^2(t), \quad i = 1, \dots, N,
\]
are the limiting processes (in the Laguerre case) with parameter $\alpha = 1/2$.

{\rm(ii)} Let $(y_1(t), \dots, y_{2N+1}(t))$ be the limiting processes in the Gaussian case with symmetric initial condition $a_1 \le a_2 \le \cdots \le a_{2N+1}$. Then $(y_1(t), \dots, y_{2N+1}(t))$ are also symmetric for $t \ge 0$, and 
\[
	x_i(t) = \frac12 y_{N+1+i}^2(t), \quad i = 1, \dots, N,
\]
are the limiting processes with parameter $\alpha = 3/2$. 
\end{lemma}
\begin{proof}
We observe that $a_1 \le a_2 \le \cdots \le a_{N}$ are symmetric, if and only if for odd $1\le k \le N$,
\[
	e_k(a_1, a_2, \dots, a_{N}) = 0.
\]
Based on that observation, it is straightforward to see that 
\[
	(c_1, \dots, c_{N}) = (a_1, \dots, a_{N}) \boxplus_{N} (b_1, \dots, b_{N})
\]
are symmetric, provided that both $(a_1, \dots, a_{N})$ and $(b_1, \dots, b_{N})$ are symmetric. 

Let us prove (i). Since 
\[
	(y_1(t), \dots, y_{2N}(t)) = (a_1, \dots, a_{2N}) \boxplus_{2N} \sqrt{t} (z_{1,2N}^{(H)}, \dots, z_{2N, 2N}^{(H)}),
\]
the processes $(y_1(t), \dots, y_{2N}(t))$ are symmetric. Because of the symmetry, it holds that
\[
	\prod_{i=1}^{2N}(x - y_i(t)) = \prod_{i=1}^{N} (x^2 - y_{N+i}(t)^2).
\]
Consequently, 
\[
	e_{2k}(y_1(t), \dots, y_{2N}(t)) = (-1)^k e_k(y_{N+1}(t)^2,\dots, y_{2N}(t)^2).
\]
It then follows from the definition of $x_i(t)$ that for $k = 1, \dots, N$,
\[
	e_k(x_1(t), \dots, x_N(t)) = (-2)^{-k} e_{2k}(y_1(t), \dots, y_{2N}(t)).
\]

Next we calculate the derivative of $e_k(x_1(t), \dots, x_N(t))$,
\begin{align*}
	&\frac{de_k(x_1(t), \dots, x_N(t))}{dt} \\
	&= (-2)^{-k} \frac{de_{2k}(y_1(t), \dots, y_{2N}(t))}{dt}\\
	&= (-2)^{-k} (-1) \frac{(2N-2k+1)(2N-2k+2)}{2}e_{2k-2}(y_1(t), \dots, y_{2N}(t))\\
	&=(N-k+1)(N-k+\frac12) e_{k-1}(x_1(t), \dots, x_N(t)).
\end{align*}
This implies that $(x_1(t), \dots, x_N(t))$ are the limiting processes with parameter $\alpha = \frac12$. The proof of (ii) is exactly the same.
\end{proof}

We arrive at an expression of the limiting processes in terms of the initial data, zeros of Hermite polynomials and zeros of Laguerre polynomials, which is similar to that in the random matrix regime (cf.\ \cite[Theorem 1.3]{Voit-Woerner-2022b}) and that in the high temperature regime (\cite[Remark V.3]{NTT-2025}). 

\begin{theorem}
Let $\alpha > N - \frac12$. Let $0 \le a_1 \le a_2 \le \cdots \le a_N$ be a given initial condition. Let 
\[
	(y_1(t), \dots, y_{2N}(t)) = (-\sqrt{a_N}, -\sqrt{a_{N-1}}, \dots, \sqrt{a_{N-1}}, \sqrt{a_N}) \boxplus_{2N} \sqrt t (z_{1, 2N}^{(H)}, \dots, z_{2N, 2N}^{(H)})
\]
be the limiting processes in the Gaussian case. Then the limiting processes $x_1(t), \dots, x_N(t)$ can be expressed as 
\[
	(x_1(t), \dots, x_N(t)) = \frac12(y_{N+1}(t)^2, \dots, y_{2N}(t)^2) \boxplus_N t (z_{1, N}^{(\alpha - N + \frac12)}, \dots, z_{N, N}^{(\alpha - N + \frac12)}). 
\]
\end{theorem}
\begin{proof}
	First, by Lemma~\ref{lem:Claim3}(i), the processes $\frac12(y_{N+1}(t)^2, \dots, y_{2N}(t)^2)$ are the limiting processes in the Laguerre case with parameter $1/2$ under the initial condition $a_1, \dots, a_N$. Second, Lemma~\ref{lem:Claim1} states that $t (z_{1, N}^{(\alpha - N + \frac12)}, \dots, z_{N, N}^{(\alpha - N + \frac12)})$ are the limiting processes with parameter $(\alpha - N + \frac12)$ under the zero initial condition. Finally, the conclusion follows immediately from Lemma~\ref{lem:Claim2}. The proof is complete.
\end{proof}

\subsection{Central limit theorems for the empirical distributions}
By a moment method, we can obtain the LLN and CLT for the empirical measure processes in similar forms with the high temperature regime. Let us introduce a CLT for the empirical distributions of beta Laguerre ensembles by using orthogonal polynomials. Let 
\[
	L_N = \frac1N \sum_{i=1}^N \delta_{\lambda_i}
\]
be the empirical distribution of the eigenvalues of beta Laguerre ensembles~\eqref{bLE}. Denote by 
\[
	\mu = \frac1N \sum_{i=1}^N \delta_{z_{i, N}^{(\alpha)}}
\]
the limiting measure in the freezing regime. We consider duals of Laguerre polynomials, denoted by $\{q_n\}_{n=0}^{N-1}$, which are orthogonal with respect to the probability measure 
\[
	\mu^* = \frac{1}{N(\alpha + N - 1)}\sum_{i=1}^N z_{i, N}^{(\alpha)}\delta_{z_{i, N}^{(\alpha)}}.
\]
(See Example~{\rm\ref{ex:Laguerre}}.) Let $Q_n$ be a primitive of $q_n$. Then similar to the high temperature regime \cite{NTT-2023}, we can establish the following.
\begin{theorem}\label{thm:bLE-CLT}
As $\beta \to \infty$, 
\[
	\sqrt {\beta N/2} \left( \bra{L_N, Q_n} - \bra{\mu, Q_n}\right)_{n=0}^{N-1} \dto \Normal_N\left(0, \diag\left(\frac{(\alpha + N - 1)}{n+1} \bra{q_n, q_n}_{\mu^*}\right)\right).
\]
\end{theorem}

Let $\hat q_n = q_n / \sqrt{\bra{q_n, q_n}_{\mu^*}}, n = 0, \dots, N-1$ be orthogonal polynomials with respect to $\mu^*$. Define the orthogonal matrix $\Q$ by 
\[
	\Q = \frac{1}{\sqrt{N(N+\alpha-1)}}\left(\sqrt{z_{i, N}^{(\alpha)}} \hat q_n(z_{i, N}^{(\alpha)}) \right)_{n = 0, \dots, N-1; i = 1, \dots, N}.
\]
\begin{corollary}
	As $\beta \to \infty$, 
\[
	\left\{\sqrt {2\beta}\left(\sqrt {\lambda_i} - \sqrt{z_{i, N}^{(\alpha)}}\right) \right\}_{n=0}^{N-1} \dto \Normal_N(0, \Sigma),
\]
where the limiting covariance matrix $\Sigma$ satisfies $\Q\Sigma \Q^\tran = \diag(\frac1{n+1})_{n=0}^{N-1}.$

\end{corollary}
\begin{proof}
We use the same idea as in the proof of Corollary~\ref{coro:G}.  Let $\xi_i = \sqrt{\lambda_i}$ and $u_i = \sqrt{z_{i, N}^{(\alpha)}}$. Then the following LLN and CLT hold
\[
	(\xi_i)_{i=1}^N \Pto \left(u_i \right)_{i=1}^N,
\]
\[
	\left\{\sqrt{2\beta} \left(\xi_i - u_i \right) \right\}_{i=1}^N \dto \{\eta_i\}_{i=1}^N,
\]
where $\{\eta_i\}_{i=1}^N$ are jointly Gaussian with mean zero and covariance matrix $\Sigma$. 

Let $\hat Q_n$ be a primitive of $\hat q_n$. We begin with the following expression 
\begin{align*}
\sqrt{\beta N/2}\left(\langle L_N, \hat{Q}_n \rangle - \langle \mu^*, \hat{Q}_n \rangle\right) &= \frac{\sqrt{\beta / 2}}{\sqrt N} \sum_{i=1}^N  \left(\hat Q_n(\xi_i^2) - \hat Q_n(u_i^2)\right) \\
&=\frac{\sqrt{\beta / 2}}{\sqrt N} \sum_{i=1}^N  2\gamma_{i} \hat q_n(\gamma_{i}^2) \left(\xi_i - u_i\right)\\
&=\frac{1}{\sqrt N} \sum_{i=1}^N  \gamma_{i} \hat q_n(\gamma_{i}^2) \sqrt{2\beta}\left(\xi_i - u_i\right),
\end{align*}
for some $\gamma_i$ between $\xi_i$ and $u_i$, by applying the mean value theorem. Then the above LLN and CLT imply that 
\[
	\sqrt{\beta N/2}\left(\langle L_N, \hat{Q}_n \rangle - \langle \mu^*, \hat{Q}_n \rangle\right)  \dto \frac{1}{\sqrt N} \sum_{i=1}^N   u_i \hat q_n(u_i^2)\eta_i.
\]
The joint convergence also holds. Consequently, as column vectors, 
\[
	\left\{\sqrt{\beta N/2}\left(\langle L_N, \hat{Q}_n \rangle - \langle \mu^*, \hat{Q}_n \rangle\right)\right\}_{n=0}^{N-1} \dto \sqrt{N+\alpha-1} \Q \eta, \quad \eta = (\eta_1, \dots, \eta_N)^\tran.
\]
In addition, we rewrite the CLT in Theorem~\ref{thm:bLE-CLT} by using primitives of orthonormal polynomials $\hat q_n$ as
\[
	\sqrt {\beta N/2} \left( \bra{L_N, \hat Q_n} - \bra{\mu, \hat Q_n}\right)_{n=0}^{N-1} \dto \Normal_N\left(0, \diag\left(\frac{(\alpha + N - 1)}{n+1} \right)\right).
\]
Therefore, the covariance matrix $\Q \Sigma \Q^\tran$ of $\Q \eta$ coincides with $\diag(\frac1{n+1})_{n=0}^{N-1}$, completing the proof.
\end{proof}

\bigskip
\noindent\textbf{Acknowledgement.} The authors would like to thank Professor Peter J. Forrester for suggesting us using a stochastic approach to investigate the freezing regime.

\appendix
\section{Orthogonal polynomials and their duals}
\label{sect:duals}
This section deals with orthogonal polynomials, Jacobi matrices and their spectral measures. We aim to introduce orthogonal polynomials with respect to discrete probability measures of the form  
\[
	 \sum_{i=1}^N \frac{1}{N} \delta_{z_{i, N}^{(H)}}, \quad	 \frac{1}{const}\sum_{i=1}^N z_{i, N}^{(\alpha)} \delta_{z_{i, N}^{(\alpha)}},
\] 
in terms of dual polynomials,
where $z_{i, N}^{(H)}$ and $z_{i, N}^{(\alpha)}$ are zeros of Hermite polynomials, and of Laguerre polynomials, respectively.

Let $J$ be a finite Jacobi matrix, a symmetric tridiagonal  matrix of the form
\begin{equation}
	J = \begin{pmatrix}
		a_1		&b_1\\
		b_1		&a_2		&b_2\\
		&\ddots	&\ddots	&\ddots\\
		&&b_{N-1}	&a_N
	\end{pmatrix},
\end{equation}
where $a_1, \dots, a_N \in \R; b_1, \dots, b_{N-1}>0$. Then $J$ has $N$ distinct eigenvalues $\lambda_1, \dots, \lambda_N$ with corresponding normalized eigenvectors $v_1, \dots, v_N$.
The spectral measure of $J$ is the probability measure $\mu$ on $\R$ satisfying 
\begin{equation}\label{sp-moment}
	\bra{\mu, x^n} = \int_{\R} x^n d\mu(x) = J^n(1,1), \quad n = 0,1,\dots. 
\end{equation}
It follows from the spectral decomposition of $J$ that the spectral measure $\mu$ has the form
\[
	\mu = \sum_{i=1}^N w_i \delta_{\lambda_i}, \quad w_i = |v_i(1)|^2,
\]
which is a discrete probability measure supported on the eigenvalues of $J$.

Define a sequence of monic polynomials $p_0, \dots, p_N$ by the three-term recurrence relation
\begin{align*}
	&p_0(x) = 1, \quad p_1(x) = x - a_1,\\
	&p_{n+1}(x) = xp_n(x) - a_{n+1} p_n(x) - b_n^2 p_{n-1}(x), \quad n = 1, \dots, N-1.
\end{align*}
Then $\{p_i\}_{i=0}^{N-1}$ are orthogonal in $L^2(\mu)$ with orthogonal relations
\[
	\int_\R p_n(x) p_m(x) d\mu(x)  = \delta_{mn} \prod_{i=1}^n b_i^2,\quad 0 \le m, n \le N - 1,
\]
in other words,
\[
	\sum_{i=1}^N w_i p_n(\lambda_i) p_m(\lambda_i) = \delta_{mn} \prod_{i=1}^{n} b_i^2.
\]
Here $\delta_{mn} = 1$, if $m = n$ and $\delta_{mn}=0$, otherwise.
Note that $p_N(x) = \det(x - J)= \prod_{i=1}^N(x - \lambda_i)$ is a zero function in $L^2(\mu)$. For $n =0, 1,\dots, N-1$, let
\[
	\tilde p_n  = p_n / \sqrt{h_n}, \quad (h_n := b_1^2 \cdots b_n^2).
\]
Then $\{\tilde p_n\}_{n=0}^{N-1}$ are orthonormal polynomials. Those polynomials satisfy the following three-term recurrence relation
\[
	b_{n+1}\tilde p_{n+1} = x\tilde p_n - a_{n+1}\tilde p_n - b_n\tilde p_{n-1}, \quad n = 0, \dots, N-1,
\]
$(b_0 := 0, b_N := 1)$. We rewrite it in a matrix form as
\begin{equation}
	\begin{pmatrix}
		a_1		&b_1\\
		b_1		&a_2		&b_2\\
		&\ddots	&\ddots	&\ddots\\
		&&b_{N-1}	&a_N
	\end{pmatrix}
	\begin{pmatrix}
		\tilde p_0 \\
		\tilde p_1 \\
		\vdots\\
		\tilde p_{N-1}
	\end{pmatrix}
	=
	\begin{pmatrix}
		x\tilde  p_0 \\
		x\tilde  p_1 \\
		\vdots\\
		x\tilde  p_{N-1} - \tilde p_N
	\end{pmatrix}.
\end{equation}
We deduce that  
\[
	(\tilde p_0(\lambda_i), \dots, \tilde p_{N-1}(\lambda_i))^\top,
\]
is an eigenvector of $J$ with respect to $\lambda_i$,
and thus, the normalized eigenvector with respect to the eigenvalue $\lambda_i$ is given by
\begin{equation}\label{normalized-ev}
	v_i = (\tilde p_0(\lambda_i), \dots, \tilde p_{N-1}(\lambda_i))^\top/\sqrt{\sum_{j=0}^{N-1} \tilde p_j(\lambda_i)^2}.
\end{equation}
In particular, the weights $w_i$ in the expression of the spectral measure can be expressed as
\begin{equation}
	w_i = \frac{1}{\sum_{j=0}^{N-1} \tilde p_j(\lambda_i)^2},
\end{equation}
which can be further simplified to be 
\begin{equation}
	w_i = \frac{h_{N-1}}{p_{N-1}(\lambda_i) p_N'(\lambda_i)}
\end{equation}
by using the Christoffel--Darboux formula.

The dual of $J$ is defined to be the Jacobi matrix $J^*$ by reversing the sequences $\{a_n\}_{n=1}^N$ and $\{b_n\}_{n=1}^{N-1}$,
\[
	J^* = \begin{pmatrix}
		a_N		&b_{N-1}\\
		b_{N-1}		&a_{N-1}		&b_{N-2}\\
		&\ddots	&\ddots	&\ddots\\
		&&b_{1}	&a_1
	\end{pmatrix}.
\] 
The polynomials $\{q_n\}_{n=0}^{N-1}$ defined by
\begin{align*}
	&q_0 = 1, q_1 = x - a_N,\\
	&q_{n+1} = x q_n - a_{N-n} q_n - b_{N-n}^2 q_{n-1}, \quad n =1, \dots, N - 2,
\end{align*}
are called dual polynomials of $\{p_n\}_{n=0}^{N-1}$. Note that $\{q_n\}_{n=0}^{N-1}$ are orthogonal with respect to the spectral measure $\mu^*$ of $J^*$ which can be expressed as
\[
	\mu^* = \sum_{i=1}^N w_i^* \delta_{\lambda_i}, \quad w_i^* = |v_i(N)|^2,
\]
because $(J^*)^n(1, 1) = J^n(N, N), n = 0,1,\dots.$
Formula~\eqref{normalized-ev} leads to
\begin{equation}\label{weights*}
	w_i^*  = \frac{\tilde p_{N-1}^2(\lambda_i)}{\sum_{j=0}^{N-1} \tilde p_j^2(\lambda_i)} = \tilde p_{N-1}^2(\lambda_i) w_i = \frac{h_{N-1}\tilde p_{N-1}^2(\lambda_i)}{p_{N-1}(\lambda_i) p_N'(\lambda_i)} = \frac{p_{N-1}(\lambda_i)}{p_N'(\lambda_i)}.
\end{equation}

In the infinite case, 
\[
	J = \begin{pmatrix}
		a_1		&b_1\\
		b_1		&a_2		&b_2\\
		&\ddots	&\ddots	&\ddots\\
	\end{pmatrix},
	\quad a_i \in \R; b_j > 0,
\]
monic polynomials $\{p_n\}_{n=0}^\infty$ are defined by the three-term recurrence relation up to infinity. They are orthogonal with respect to any probability measure $\mu$ satisfying the moment relation~\eqref{sp-moment}. In case the probability measure $\mu$ is unique, it is called the spectral measure of $J$. A sufficient condition for the unicity is given by 
\[
	\sum_{n=1}^\infty \frac{1}{b_n} = \infty,
\]
(see \cite[Corollary
3.8.9]{Simon-book-2011})
We also call $J$ the Jacobi matrix of $\mu$. In this infinite case, we consider the dual polynomials of the first $N$ polynomials $\{p_i\}_{i=0}^{N-1}$ which are denoted by $\{q_i^{(N)}\}_{i=0}^{N-1}$ or simply as $\{q_i\}_{i=0}^N$ when it is clear from the context.

\begin{example}[
Hermite polynomials]\label{ex:Hermite} (Probabilist's) Hermite polynomials are monic polynomials orthogonal with respect to the standard Gaussian measure $ \frac{1}{\sqrt{2\pi}} e^{-\frac{x^2}2} dx$. They can be defined recursively by
\begin{align*}
	&H_0(x) = 1, \quad H_1(x) = x,\\
	&H_{n+1}(x) = xH_n(x)  - n H_{n-1}(x), \quad n \ge 1.
\end{align*}
Note that $H_n(x)$ has an explicit expression as follows
\begin{equation}\label{HeN}
{H}_{n}(x) =\sum_{m \le n/2}  \frac{(-1)^m}{2^m} \frac{n!}{m!(n-2m)!} x^{n-2m}.
\end{equation}
In terms of Jacobi matrices, the standard Gaussian measure $ \frac{1}{\sqrt{2\pi}} e^{-\frac{x^2}2} dx$ is the spectral measure of the following infinite Jacobi matrix 
\[
	J^{(H)} = \begin{pmatrix}
		0		&1\\
		1		&0		&\sqrt 2\\
		&\sqrt 2		&0		&\sqrt 3\\
		&&\ddots	&\ddots	&\ddots\\
	\end{pmatrix}.
\]

Define the finite Jacobi matrix $J_N^{(H)}$ by
\begin{equation}\label{JNH}
	J_N^{(H)} = \begin{pmatrix}
		0		&1\\
		1		&0		&\sqrt 2\\
		&\ddots	&\ddots	&\ddots\\
		&&\sqrt{N-1}	&0
	\end{pmatrix}.
\end{equation}
Then $H_N(x) = \det(x - J_N^{(H)})$. Consequently, the eigenvalues $\{z_{i,N}^{(H)}\}_{i=1}^N$ of $J_H^{(N)}$ are zeros of the $N$th Hermite polynomial $H_N(x)$. For this finite matrix, dual polynomials $\{q_i\}_{i=0}^{N-1}$ are orthogonal w.r.t. (see \cite{Vinet-Zhedanov-2004})
\[
	\mu^* = \sum_{i=1}^N \frac1N \delta_{z^{(H)}_{i,N}}.
\]
This has been derived by using equation~\eqref{weights*} with the help of the relation $H'_N(x) = N H_{N-1}(x)$, 
\[
	w_i^* = \frac{H_{N-1}(z^{(H)}_{i,N})}{H_N'(z^{(H)}_{i,N})} = \frac 1N.
\]
\end{example}

\begin{example}[(Generalized) Laguerre polynomials]
\label{ex:Laguerre} We consider monic polynomials $L_n^{(\alpha)}(x)$ which are orthogonal with respect to the weight 
\[
	\frac{1}{\Gamma(\alpha)} x^{\alpha - 1} e^{-x}, \quad x > 0,
\]
which is the density of the gamma distribution with parameters $(\alpha, 1)$, 
where $\alpha > 0$ is a parameter. (Note that usual Laguerre polynomials have leading coefficient $(-1)^n/n!$.) They satisfy the three-term recurrent relation 
\[
	L_{n+1}^{(\alpha)}(x) = (x - (\alpha + 2n)) L_n^{(\alpha)}(x) - n(\alpha + n-1) L_{n-1}^{(\alpha)}(x), \quad n \ge 1,
\]
with $L_0^{(\alpha)}(x) = 1, L_1^{(\alpha)}(x) = x - \alpha$. In terms of Jacobi matrix, it means that the above gamma density is the spectral measure of 
\begin{align*}
	J^{(\alpha)} &= \begin{pmatrix}
		\alpha	& \sqrt{\alpha} \\
		\sqrt{\alpha} 	&\alpha + 2	& \sqrt{2}\sqrt{ \alpha + 1} 	\\
		&&\ddots	&\ddots	&\ddots 
	\end{pmatrix} \\
	&=	\begin{pmatrix}
		\sqrt{\alpha}	\\
		\sqrt 1	&\sqrt{\alpha + 1}	\\
		&\ddots	&\ddots	
	\end{pmatrix}
	\begin{pmatrix}
		\sqrt{\alpha}	&\sqrt1	\\
			&\sqrt{\alpha + 1}	&\sqrt{2}	\\
			&&\ddots	&\ddots	
	\end{pmatrix}.
\end{align*}
Note that Laguerre polynomials have a closed form as 
\begin{equation}
\label{closedform}
	L_{n}^{(\alpha)}(x) = \sum_{k=0}^n \frac{(-1)^k}{k!} \left[\prod_{i=0}^{k-1}(N-i)(N-i+\alpha -1) \right] x^{n - k},
\end{equation}
which can be derived by applying Leibniz’s theorem for differentiation of a product to Rodrigues’ formula
\[
L_{n}^{(\alpha )}(x)=(-1)^n {x^{-\alpha + 1 }e^{x}}\frac{d^{n}}{ dx^{n}}\left(e^{-x}x^{n+\alpha-1}\right).
\]

Let $J_{\alpha,N} $ be the finite Jacobi matrix
\[
	J_{\alpha,N} = \begin{pmatrix}
		\alpha	& \sqrt{\alpha} \\
		\sqrt{\alpha} 	&\alpha + 2	& \sqrt{2}\sqrt{ \alpha + 1} 	\\
		&\ddots	&\ddots	&\ddots \\
		&	&\sqrt{N-1} \sqrt{\alpha + N - 2}&\alpha + 2(N-1)
	\end{pmatrix} .
\]
We express its dual as
\begin{align*}
	&J_{\alpha,N}^* = \begin{pmatrix}
	\alpha + 2(N-1)	&\sqrt{N-1} \sqrt{\alpha + N - 2}\\
	\sqrt{N-1} \sqrt{\alpha + N - 2}	&\alpha + 2(N-1)&\sqrt{N-2} \sqrt{\alpha + N - 3}	\\
		&\ddots	&\ddots	&\ddots \\
	&&	\sqrt{\alpha} 	&\alpha \\
	\end{pmatrix} \\
	&=	\begin{pmatrix}
		\sqrt{\alpha+N-1}	&\sqrt{N-1}	\\
			&\sqrt{\alpha+N-2}	&\sqrt{N-2}	\\
			&&\ddots	&\ddots	\\
			&&&\sqrt \alpha
	\end{pmatrix} \times \text{its transpose.}
\end{align*}

Here are some properties needed in this paper.
\begin{itemize}
	\item[(i)] The zeros  $z_{i,N}^{(L)}$ of $L_N^{(\alpha)}(x)$ are the eigenvalues of $J_{\alpha, N}$ or of $J_{\alpha, N}^*$.
	
	\item[(ii)] The spectral measure of $J^*_{\alpha, N}$ is given by 
\[	
	\mu^* = \frac{1}{N(\alpha + N - 1)}\sum_{i=1}^N z_{i, N}^{(L)} \delta_{z_{i, N}^{(L)}}.
\]
To see this, we use the relation $x (L_N^{(\alpha)})'(x) = N L_N^{(\alpha)}(x) + N(\alpha+N-1) L_{N-1}^{(\alpha)}(x)$ to deduce that 
\[
	w_i^* = \frac{L_{N-1}^{(\alpha)}(\lambda_i)}{(L^{(\alpha)}_N)'(\lambda_i)} = \frac{\lambda_i}{N(\alpha + N - 1)}, \quad (\lambda_i = z_{i, N}^{(L)}).
\]

\item[(iii)] We remark that
\begin{align*}
	&
	\begin{pmatrix}
		\sqrt{\alpha+N-1}	\\
		\sqrt{N-1}	&\sqrt{\alpha +N-2}	\\
		&\ddots	&\ddots	\\
		&&1	&\sqrt \alpha
	\end{pmatrix}\begin{pmatrix}
		\sqrt{\alpha+N-1}	&\sqrt{N-1}	\\
			&\sqrt{\alpha+N-2}	&\sqrt{N-2}	\\
			&&\ddots	&\ddots	\\
			&&&\sqrt \alpha
	\end{pmatrix}\\
	&= \begin{pmatrix}
	\alpha + N-1	&\sqrt{N-1} \sqrt{\alpha + N - 1}\\
	\sqrt{N-1} \sqrt{\alpha + N - 1}	&\alpha + 2N-3&\sqrt{N-2} \sqrt{\alpha + N - 2}	\\
		&\ddots	&\ddots	&\ddots \\
	&&	\sqrt{\alpha+1} 	&\alpha+1 \\
	\end{pmatrix} =: J_N^{(L)}\\
\end{align*}
also have the same eigenvalues with $J_{\alpha, N}^*$. This matrix was defined in \eqref{JNL} as the limit in the freezing limit of beta Laguerre ensembles. Its spectral measure is given by 
\[
	\frac{1}{N} \delta_{z_{i, N}^{(L)}}.
\]

\end{itemize}
\end{example}

\section{Finite free convolutions and the Fourier transform}

We next define the notion of ``Fourier transform" on the polynomials and study some of its properties. 
For a polynomial
$p$
of degree 
$N$, 
we can uniquely find a differential polynomial 
$\widehat{p}(D)$ (of degree $N$)
with 
\begin{eqnarray*}
\widehat{p}(D) x^N
=
p(x),
\end{eqnarray*}
where 
$D$
is the differential operator with respect to the variable 
$x$. 
We 
call 
$\widehat{p}$
the {\bf finite free Fourier transform (FFF)}
of 
$p$.
In fact, for 
the polynomial 
$
p(x)
=
\sum_{k=0}^N
(-1)^k
\alpha_k
x^{N-k}
$, 
its FFF
$\widehat{p}(D)$
is given explicitly by 
\begin{eqnarray*}
\widehat{p}(D)
=
\sum_{k=0}^N
\frac { (-1)^k }{k!}
\frac {\alpha_k}{
\binom{N}{k}
}
D^k.
\end{eqnarray*}
We write 
\[
	\hat p(D) \overset{N}{=} \hat q(D)
\]
if the coefficients of $D^n$ are the same for all $n \le N$.
Then if follows directly from the definition that $r = p \boxplus_N q$, if and only if 
\begin{equation}
\widehat{r
}
(D)
\overset{N}{=}
\widehat{p} (D)
\widehat{q} (D).
\label{FFCFFF}
\end{equation}

Let 
$p(x) = 
\prod_{i=1}^N
(x - a_i)$
be a monic polynomial with roots
$a_1, \dots, a_N$.
We define
$p_t (x)$
to be the one with roots 
$t a_1, \dots, t a_N$,
\begin{eqnarray*}
p_t (x)
=
\prod_{i}^N
(x - t a_i), 
\quad
t \in {\bf C}.
\end{eqnarray*}
Since 
the corresponding elementary symmetric functions satisfy 
\[
	e_k(ta_1, \dots, t a_{N}) = t^k e_k(a_1, \dots, a_N),\quad k = 0, \dots, N,
\] 
we obtain that  
\begin{equation}
\widehat{p_t}(D)
=
\widehat p (t D).
\label{pt}
\end{equation}
When $p = H_N$ is the $N$th Hermite polynomial, it follows from the explicit formula~\eqref{HeN} that
\[
	\widehat{p}(D) =  \sum_{n \le N/2} \frac{(-1)^n D^{2n}}{2^n n!} \overset{N}{=}  \sum_{n=0}^\infty \frac{(-1)^n D^{2n}}{2^n n!}  = e^{-\frac12 D^2}.
\]
And thus, with $H_{N, t}$ denoting the polynomial with roots being $t$ times the roots of $H_N$, its Fourier transform is given by 
\[
	\widehat{H_{N, t}} (D) =\widehat {p_t} \overset{N}{=} e^{-\frac12 t^2 D^2}.
\]
A direct consequence of this formula is that 
\[
	H_{N, t} \boxplus_N H_{N, s} = H_{N, \sqrt{t^2 + s^2}}.
\]

We aim to give an alternative proof of Lemma~\ref{lem:finite-convolution} by using the Fourier transform.

\begin{lemma}
For $a_1, \dots, a_N \in \R$, define $g_k(t)$ recursively by $g_0(t)=1, g_1(t)= e_1(a_1,\ldots,a_N),$
\begin{equation}\label{ODE-gka}
	g_k(t) = e_k(a_1, \dots, a_N) - \frac{(N-k+1)(N-k+2)}{2}  \int_0^t g_{k-2} (s) ds, \quad (k \ge 2).
\end{equation}
Let 
$y_1 (t) \le \cdots \le y_N (t)$
be defined by the relations
\begin{eqnarray*}
e_k (y_1(t), \dots, y_N(t)) = g_k(t), 
\quad
k = 1, \dots, N,
\end{eqnarray*}
and let 
$p_t$
be the monic polynomial which has 
$(y_1(t), \dots, y_N(t))$
as its roots. Then 
\[
	\widehat{p_t}(D) = \widehat{p_0}(D) e^{-\frac12 t D^2}.
\]
Consequently, $p_t$ is the finite free convolution of $p_0 = \prod_i(x - a_i)$ and $H_{N, \sqrt t}$. In other words, 
\[
	(y_1(t), \dots, y_N(t)) = (a_1, \dots, a_N) \boxplus_N \sqrt t (z_{1, N}^{(H)}, \dots, z_{N, N}^{(H)}).
\]
\end{lemma}
\begin{proof}
Let 
$p_t$
be the polynomial with roots $y_i(t)$. Then
\[
	p_t(x) = \prod_{i=1}^N (x - y_i(t)) = \sum_{k=0}^N (-1)^k g_k(t) x^{N-k},
\]
and thus
\begin{eqnarray*}
\widehat{p}_t (D)
&=&
\sum_{k=0}^N
\frac {(-1)^k}{k!}
\frac {g_k(t)}{\binom{N}{k}}
D^k.
\end{eqnarray*}
Differentiating both sides and using the relation 
\[
	\frac {d}{dt}
g_k(t)
=
- \frac{(N-k+1)(N-k+2)}{2} g_{k-2}(t),
\]
we deduce that 
\begin{align*}
	\frac {d}{dt}
\widehat{p}_t (D) &= \sum_{k=2}^N
\frac {(-1)^k}{k!}
\frac {g_k'(t)}{\binom{N}{k}}
D^k \\
	&=-\frac12 \sum_{k=2}^N\frac {(-1)^k}{k!}
\frac {(N-k+1)(N-k+2)g_{k-2}(t)}{\binom{N}{k}}
D^k\\
&=-\frac {1}2  \sum_{k=2}^N\frac {(-1)^{k-2}}{(k-2)!}
\frac {g_{k-2}(t)}{\binom{N}{k-2}}
D^k\\
&=-\frac{D^2}2 \sum_{l=0}^{N-2}\frac {(-1)^{l}}{l!}
\frac {g_{l}(t)}{\binom{N}{l}}
D^l, \quad (l = k-2)\\
&\overset{N}{=}- 
\frac {D^2}{2}
\widehat{p}_t (D).
\end{align*}

Now the solution to the ODE
\[
	\frac {d}{dt}
\widehat{p}_t (D) = -\frac {D^2}{2}
\widehat{p}_t (D),
\]
(in the sense of formal power series)
with initial condition $\widehat{p_0}(D)$ is unique and is given by 
\[
	\widehat{p}_t (D) = 	\widehat{p_0}(D) e^{-\frac12 t D^2},
\]
implying that $p_t$ is the finite free convolution of $p_0$ and $H_{N, \sqrt t}$. The proof is complete.
\end{proof}

Next we turn to the Laguerre case. It follows from the explicit formula~\eqref{closedform} for Laguerre polynomials that 
\[
	\widehat{L_N^{(\alpha)}}(D) =  \sum_{k=0}^N \frac{(-1)^k}{k!} \left[\prod_{i=0}^{k-1}(N-i+\alpha -1) \right] D^{ k}  \overset N{=} (1 - D)^{N + \alpha - 1}.
\]
Consequently, we immediately get a static version of Lemma~\ref{lem:Claim2},
\[
	L_N^{(\alpha_1)} \boxplus_N L_N^{(\alpha_2)} = L_N^{(N+\alpha_1+\alpha_2 -1)}.
\]
It also follows that 
\[
	\frac{d}{dt} \widehat{L_N^{(\alpha)}}(t D) = -(N + \alpha - 1) D \widehat{L_{N-1}^{(\alpha)}}(tD),
\]
which provides a relation between the $N$th and $(N-1)$st Laguerre polynomials.

\bibliographystyle{spmpsci}
\bibliography{refs}

\end{document}